\documentclass[11pt, twoside]{amsart}

\title[Reflective centers of module categories]{Reflective centers of module categories \\and  quantum K-matrices}

\author{Robert Laugwitz}
\address{School of Mathematical Sciences,
University of Nottingham, University Park, Nottingham, NG7 2RD, UK}
\email{robert.laugwitz@nottingham.ac.uk}

\author{Chelsea Walton}
\address{Department of Mathematics, Rice University,
P.O. Box 1892, Houston, TX 77005-1892, USA}
\email{notlaw@rice.edu}

\author{Milen Yakimov}
\address{Department of Mathematics, Northeastern University,
360 Huntington Ave., Boston, MA 02115, USA
and International Center for Mathematical Sciences, Institute of Mathematics and Informatics,
Bulgarian Academy of Sciences, 
Acad. G. Bonchev Str., Bl. 8, 
Sofia 1113, Bulgaria
}
\email{m.yakimov@northeastern.edu}

\usepackage{mathabx}
\usepackage{import}
\usepackage{amsmath}
\usepackage{amsfonts, stmaryrd}
\usepackage{amsthm}
\usepackage{amssymb,bbm}
\usepackage{dsfont}
\usepackage[alphabetic, initials, nobysame]{amsrefs}
\usepackage[english]{babel}
\usepackage{url}
\usepackage{fancyhdr}
\usepackage{graphicx}
\usepackage{verbatim}
\usepackage[normalem]{ulem}
\usepackage[shortlabels]{enumitem}
\newcommand{\stkout}[1]{\ifmmode\text{\sout{\ensuremath{#1}}}\else\sout{#1}\fi}
\usepackage[
colorlinks=true,
linkcolor=black, 
anchorcolor=black,
citecolor=black,
urlcolor=black, 
]{hyperref}
\usepackage{cleveref} %%This package allows references to automatically pick up Lemma/Proposition/Equation etc
\usepackage[all,cmtip]{xy}
\usepackage{geometry}
\usepackage{microtype} %%Reduces bad spacing
\usepackage[dvipsnames]{xcolor}

\usepackage{cancel} %%%Remove after editing is done

\allowdisplaybreaks

\input xy
\xyoption{all}

\definecolor{forest}{rgb}{0.0, 0.5, 0.0}

\newcommand{\bijar}[1][]{%
 \ar[#1]
 \ar@<0.7ex>@{}[#1]|-*=0[@]{\sim}} 
\usepackage{enumitem}

\usepackage{calrsfs}
\DeclareMathAlphabet{\cal}{OMS}{zplm}{m}{n}
\DeclareMathAlphabet{\mathsf}{OT1}{cmss}{m}{n} % to math mathsf thinner

\usepackage[justification=centering]{caption}

\newcommand{\leftexpsub}[3]{{\vphantom{#3}}^{#1}_{#2}{#3}}

\newcommand{\lYD}[1]{\leftexpsub{#1}{#1}{\mathsf{YD}}}

\newcommand{\ov}[1]{\overline{#1}}

\newcommand{\lmod}[1]{#1\text{-}\mathsf{mod}}

\newcommand{\lcomod}[1]{#1\text{-}\mathsf{comod}}

\newcommand{\act}{\triangleright}
\DeclareFontFamily{U}{stixscr}{}
\DeclareFontShape{U}{stixscr}{m}{n}{<-> s*[0.8] stix-mathscr}{}
\DeclareRobustCommand{\blkact}{%
  \mathrel{\text{\usefont{U}{stixscr}{m}{n}\symbol{"D1}}}%
}
\DeclareRobustCommand{\grayact}{%
  {\color{gray}\mathrel{\text{\usefont{U}{stixscr}{m}{n}\symbol{"D1}}}}%
}

\newcommand\diam{\mathbin{\text{\scalebox{0.6}{$\blacklozenge$}}}}

\newcommand{\coev}{\mathsf{coev}}

\newcommand{\Drin}{\operatorname{Drin}}
\newcommand{\ev}{\mathsf{ev}}

\newcommand{\End}{\operatorname{End}}

\newcommand{\Hom}{\operatorname{Hom}}

\newcommand{\ide}{\mathsf{Id}}

\newcommand{\isomorph}{\stackrel{\sim}{\to}}

\newcommand{\one}{\mathds{1}}

\newcommand{\ract}{\triangleleft}

\newcommand{\Vect}{\mathsf{Vec}}

\newcommand{\cC}{\cal{C}}

\newcommand{\cE}{\cal{E}}

\newcommand{\cM}{\cal{M}}
\newcommand{\cN}{\cal{N}}

\newcommand{\cZ}{\cal{Z}}

\newtheoremstyle{defstyle}% name of the style to be used
  {0.5cm}                   %Space above
  {0.5cm}                   %Space below
  {\normalfont}           %Body font
  {}     %Indent amount (empty = no indent,
  {\normalfont\bfseries}  %Thm head font
  {:}                     %Punctuation after thm head
  {0.3cm}              %Space after thm head: " " = normal interword
  {\thmname{#1}\thmnumber{ #2}\thmnote{ (#3)}}
\numberwithin{equation}{section}

\newtheorem*{rep@theorem}{\rep@title}
\newcommand{\newreptheorem}[2]{%
\newenvironment{rep#1}[1]{%
 \def\rep@title{#2 \ref{##1}}%
 \begin{rep@theorem}}%
 {\end{rep@theorem}}}
\makeatother

\newtheorem{theorem}{Theorem}[section]

\newtheorem{proposition}[theorem]{Proposition}
\newreptheorem{proposition}{Proposition}
\newtheorem{corollary}[theorem]{Corollary}
\newreptheorem{corollary}{Corollary}
\newtheorem{lemma}[theorem]{Lemma}

\newtheorem{theorem*}{Theorem}
\newreptheorem{theorem}{Theorem}

\theoremstyle{definition}
\newtheorem{definition}[theorem]{Definition}

\newtheorem{example}[theorem]{Example}
\newtheorem{remark}[theorem]{Remark}

\makeatletter              % This sequence of commands will
\let\c@equation\c@theorem  % incorporate equation numbering
\makeatother
\numberwithin{equation}{section}

\makeatletter
\@namedef{subjclassname@2020}{%
  \textup{2020} Mathematics Subject Classification}
\makeatother


\begin{document}

\maketitle

%%%%%%%%%%%%%%%%%%%%%%%
% TEMPORAL TOOLS FOR WRITING
% Comment in the final version
%%%%%%%%%%%%%%%%%%%%%%%
%\begin{center}
%\rl{[RL's edits {\tt $\backslash$rl}]}, \hspace{.21in}
%\cw{[CW's edits {\tt $\backslash$cw}]}, \hspace{.21in}
%\my{[MY's edits {\tt $\backslash$my}]}, 

%\blue{items to be filled in {\tt $\backslash$blue}},  \hspace{.5in}
%\red{math concerns in {\tt $\backslash$red}},

%[comments in brackets], \hspace{.37in}
%strike-out \stkout{text} with {\tt $\backslash$stkout}.
%\end{center}

\begin{abstract}
Our work is motivated by obtaining solutions to the quantum reflection equation (qRE) by categorical methods. To start, given a braided monoidal category $\cC$ and  $\cC$-module category $\cM$, we introduce  a version of the Drinfeld center $\cZ(\cC)$ of $\cC$ adapted for $\cM$;  we refer to this category as the {\it reflective center} $\cE_\cC(\cM)$ of $\cM$. 
Just like $\cZ(\cC)$ is a canonical braided monoidal category attached to $\cC$, we show that $\cE_\cC(\cM)$ is a canonical braided module category attached to $\cM$; its properties are investigated in detail.

Our second goal pertains to when  $\cC$ is the  category of modules over a quasitriangular Hopf algebra $H$, and $\cM$ is the  category of modules over an $H$-comodule algebra $A$. We show that the reflective center $\cE_\cC(\cM)$ here is equivalent to a category of modules over an explicit algebra, denoted by $R_H(A)$, which we call the {\it reflective algebra} of $A$. This result is akin to $\cZ(\cC)$  being represented by the Drinfeld double $\Drin(H)$ of~$H$. We also study the  properties of reflective algebras.

Our third set of results is also in the Hopf setting above. We show that reflective algebras are quasitriangular $H$-comodule algebras, and examine their corresponding quantum $K$-matrices; this yields solutions to the qRE.  We also establish that the reflective algebra $R_H(\Bbbk)$ is an initial object in the category of quasitriangular $H$-comodule algebras, 
where $\Bbbk$ is the ground field. The case when $H$ is the Drinfeld double of a finite group is illustrated.
\end{abstract}

\setcounter{tocdepth}{1}

\begin{changemargin}{2cm}{2.5cm} 
{\footnotesize \tableofcontents}
\end{changemargin}

%%%%%%%%%%%%%%%%%%%%%%%%%%
%%%%%%%%%%%%%%%%%%%%%%%%%%
%%%%%%%%%%%%%%%%%%%%%%%%%%

\section{Introduction}\label{sec:intro}

Quasitriangular Hopf algebras and their 
 (universal) quantum $R$-matrices, introduced by Drinfeld \cite{Dri}, 
play a fundamental role in many areas of mathematics and mathematical physics, such as low dimensional topology, representation theory, quantum field theory and exactly solvable models. More generally, Joyal and Street \cite{JoyalStreet} introduced the notion of braided monoidal categories, which, similarly, are central objects for the categorical foundations of numerous studies. There are well known ways to construct both structures:
\begin{enumerate}
    \item Given a Hopf algebra $H$, its Drinfeld double $\Drin(H)$ is a quasitriangular Hopf algebra with an explicit $R$-matrix.
    \item For a monoidal category $\cC$, one constructs its Drinfeld center $\cZ(\cC)$, which is a braided monoidal category.
\end{enumerate}
The two constructions work in tandem: The Drinfeld center of the 
module category  of a finite dimensional Hopf algebra $H$ is isomorphic to the module category of its Drinfeld double.  

Going a step further, on the categorical side, Brochier \cite{Brochier13} introduced the notion of a braided module category over a braided monoidal category. On the Hopf algebra side, Kolb \cite{Kolb20} defined the notion of a quasitriangular (left) $H$-comodule algebra $A$ of a quasitriangular Hopf algebra $H$; such an algebra is equipped with a (universal) quantum $K$-matrix $K \in H \otimes A$. There are broad parallels between universal quantum $R$-matrices and $K$-matrices:
\begin{enumerate}
\item[(a)] Universal quantum $R$-matrices and $K$-matrices automatically satisfy the quantum Yang--Baxter and reflection equations, respectively. 
\item[(b)] The former give rise to representations of the Artin braid groups of type $A$, while the latter give rise to representations of the Artin braid groups of type $B$.
\item[(c)] The former are used in studying exactly solvable models in statistical mechanics without boundary, while the latter are used for solving models with boundary.
\end{enumerate}
In summary, following  Balagovic's presentation \cite{Balagovic-slides} on her joint paper with  Kolb \cite{BalagovicKolb}: 

\medskip

\begin{center}
{\small
\begin{tabular}{llll}
\hline\\[-.7pc]
If you like: &&& ... then you should also like:\\[.2pc]
1. Quantum enveloping algebras  &&& 1. Quantum symmetric pairs \\[.05pc]
2. Universal quantum $R$-matrices &&& 2. Universal quantum $K$-matrices\\[.05pc]
3. The quantum Yang--Baxter equation &&& 3. The quantum reflection equation\\[.05pc]
4. Braided tensor categories &&& 4. Braided module categories\\[.4pc]
\hline
\end{tabular}
}
\end{center}
\medskip

The most important class of quasitriangular comodule algebras that was studied to date is the class of quantum symmetric pair coideal subalgebras, introduced in the foundational works of Letzter 
\cites{Letzter1,Letzter2}. These are quantum analogs of the pairs $(U(\mathfrak{k}), U(\mathfrak{g}))$ where $\mathfrak{g}$ is a complex simple Lie algebra (or more generally a symmetrizable Kac--Moody algebra) and $\mathfrak{k}$ is a symmetric Lie subalgebra (the fixed point of an involutive automorphism of $\mathfrak{g}$). The quasitriangularity for this class of comodule algebras was established by recursively building a quantum $K$-matrix using the Lusztig bar involution \cites{BaoWang18b,BalagovicKolb,BaoWang18a,AppelVlaar}. For coideal subalgebras of arbitrary Drinfeld doubles of bosonizations of Yetter--Drinfeld modules of diagonal types, where bar involutions need not exist, quantum $K$-matrices were constructed from star products on partial bosonizations of Nichols algebras \cite{KolbYakimov}.

Given the vast applications of quantum $K$-matrices and braided module categories, one can ask the following two questions:
\begin{enumerate}
\item[(Q1)] Is there a version of the Drinfeld double construction (1) with an input of an $H$-comodule algebra for a quasitriangular Hopf algebra and an output a quasitriangular $H$-comodule algebra with an explicit quantum $K$-matrix? 
\item[(Q2)] Is there a version of the Drinfeld center construction (2) with an input a module category of a braided monoidal category $\cC$ and an output a braided module category of $\cC$?
\end{enumerate}

The goals of this paper are to fully resolve both questions. This leads to strong methods for the construction of quasitriangular comodule algebras, quantum $K$-matrices, and braided module categories that can 
be applied in broad generality. 
The following table summarizes our constructions and notions: 
\medskip

\begin{center}
{\small
\begin{tabular}{llll}
\hline\\[-.7pc]
Classical constructions and notions: &&& Our constructions and notions:\\[.2pc]
5. Drinfeld centers of tensor categories &&& 5. {\it Reflective centers} of module categories \\[.05pc]
6. Yetter--Drinfeld modules  &&& 6. Doi--Hopf modules\\[.05pc]
7. Drinfeld doubles of Hopf algebras &&& 7. {\it Reflective algebras} of comodule algebras \\[.4pc]
\hline
\end{tabular}
}
\end{center}
\medskip

Before stating the precise formulation of our results, we note that all linear structures are over an  algebraically closed field $\Bbbk$. For a $\Bbbk$-algebra $A$, let $\lmod{A}$ denote the category of left $A$-modules.

\medskip

To proceed with the aims above,  take a braided monoidal category $\cC$ and a left $\cC$-module category $\cM$. 
In Definition~\ref{def:ECM}, we define the {\it reflective center of $\cM$ with respect to $\cC$},  denoted by $\cE_\cC(\cM)$, and motivated by the construction of the Drinfeld center $\cZ(\cC)$ of $\cC$. 
The following results are established for $\cE_\cC(\cM)$
in parallel of the known properties of $\cZ(\cC)$:

\vspace{7pt}

\noindent {\bf Theorem~A} (Proposition~\ref{prop:ECM-brmodcat}, Corollary~\ref{cor:ECMprops}){\bf .} Retain the notation above. Then the reflective center $\cE_\cC(\cM)$ has the following properties.
\begin{enumerate}[\upshape (a)]
\item {\it $\cE_\cC(\cM)$ is a braided left $\cC$-module category.}
\smallskip
\item  {\it $\cE_\cC(\cM)$ is abelian when $\cM$ is exact, is finite when $\cC$ is finite and $\cM$ is exact, and is semisimple when $\cC$ and $\cM$ are finite and semisimple.}
\smallskip
\item {\it $\cE_\cC(\cM)$ is also a left $\cZ(\cC)$-module category.} \qed
\end{enumerate}

\vspace{7pt}

Now for the rest of the introduction, take $\cC = \lmod{H}$ and $\cM = \lmod{A}$, for:
\begin{itemize}
\item  $H$  a finite-dimensional quasitriangular Hopf algebra over $\Bbbk$, 

\smallskip

\item $A$ a left $H$-comodule algebra over $\Bbbk$. 
\end{itemize}

\medskip

Our first main result for this Hopf setting is given below. This is achieved via Theorem~A(a) and by applying results in Section~\ref{sec:bmcfunc} on transferring a braided module category structure across an equivalence of categories.

\vspace{7pt}

\noindent {\bf Theorem B} (Lemma~\ref{lem:ECM-DH-brmod}, Proposition~\ref{prop:ECM-DH-brmod}, Theorem~\ref{thm:ECM-DH-brmod-2}){\bf .} {\it There exists a category of Doi--Hopf modules ${}_A^{\widehat{H}} \mathsf{DH}(H)$ for  a certain left $H$-module coalgebra $\widehat{H}$ defined in Section~\ref{sec:equivalence}, and a certain  algebra $R_H(A)$ defined in Section~\ref{sec:equivalence2},
such that 
\begin{equation} \label{eq:EHA-intro}
\cE_{\lmod{H}}(\lmod{A}) \; \cong \; {}_A^{\widehat{H}} \mathsf{DH}(H) \; \cong \; \lmod{R_H(A)}
\end{equation}
as braided left module categories over $\lmod{H}$. See Figure~\ref{fig:thm-sec6} for the location of these actions.} \qed

\vspace{7pt}

The $H$-module coalgebra $\widehat{H}$ is defined in \Cref{def:hatH}, and as discussed in Remark~\ref{rem:Majid},  it is a version of Majid's {\it transmuted Hopf algebra} constructed in \cite{Majid1991}. 

\medskip

 We refer to $R_H(A)$ as the {\it reflective algebra of $A$ with respect to $H$}.  It is defined as a crossed product algebra, $A \rtimes_H (\widehat{H}^*)^{\textnormal{op}}$, as described at the beginning of Section~\ref{sec:double-DH}. It plays an analogous role for reflective centers as the Drinfeld double $\Drin(H)$ of $H$ does for the Drinfeld center $\cZ(\lmod{H})$.

\begin{figure*}[h!]
{\small
\[
\begin{array}{|c|c|c|}
\hline
&&\\[-.7pc]
\xymatrix@C=1pc{
 \lmod{H} \hspace{-.2in}
&
\ar@/^.7pc/@<-3pt>[r]^(.52){\act}_(.98){\textnormal{br.}}
&
&
\hspace{-.18in}
\cE_{\lmod{H}}(\lmod{A}) 
}
&
\xymatrix@C=1pc{
 \lmod{H} \hspace{-.2in}
&
\ar@/^.7pc/@<-3pt>[r]^(.52){\blkact}_(.98){\textnormal{br.}}
&
&
\hspace{-.18in}
{}_A^{\widehat{H}} \mathsf{DH}(H) 
}
 &
 \xymatrix@C=1pc{
 \lmod{H} \hspace{-.2in}
&
\ar@/^.7pc/@<-3pt>[r]^(.52){\grayact}_(.98){\textnormal{br.}}
&
&
\hspace{-.18in}
\lmod{R_H(A)} 
}
 \\[.4pc]
\text{[Lemma~\ref{lem:ECM-DH-brmod}]}
&  \text{[Proposition~\ref{prop:ECM-DH-brmod}]}
&   \text{[Theorem~\ref{thm:ECM-DH-brmod-2}]} \\[.4pc]
\hline
\end{array}
\]
}

\caption{Isomorphic  braided module categories over $\lmod{H}$}
\label{fig:thm-sec6}
\end{figure*}
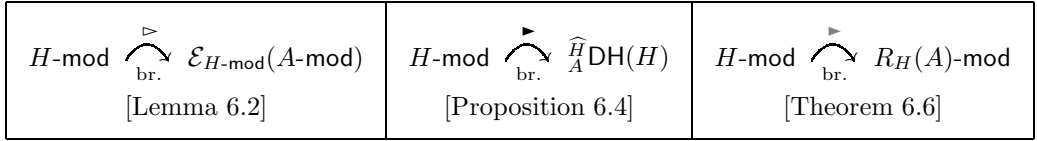

We obtain the following consequence of Theorem~B.

\vspace{7pt}

\noindent {\bf Corollary~C} (Corollary~\ref{cor:RHA-comod}){\bf.} {\it The reflective algebra $R_H(A)$ is a quasitriangular left $H$-comodule algebra, with an explicit quantum $K$-matrix  given in terms of a dual basis of $H$. }
 \qed

\vspace{7pt}

The reflective algebra $R_H(\Bbbk)$ for the canonical left $H$-comodule algebra $\Bbbk$
is of particular interest due to the following result.

\vspace{7pt}

\noindent {\bf Theorem~D} (Theorem~\ref{thm:init-obj}){\bf .} {\it The reflective algebra $R_H(\Bbbk)$ and its quantum $K$-matrix from Corollary~C is an initial object in the category of quasitriangular left $H$-comodule algebras.} \qed

\vspace{7pt}

The results above are illustrated in Section~\ref{sec:RA-ex} in the case when $H$ is the Drinfeld double of a finite group, $\Drin(G)$, and for the left $\Drin(G)$-comodule algebra $\Bbbk$.

\vspace{7pt}

Finally, using Theorem~A(c) and the braided isomorphism between $\cZ(\lmod{H})$ and $\lmod{\Drin(H)}$, we get a module category action of $\lmod{\Drin(H)}$ on $\lmod{R_H(A)}$. This yields the result below.

\vspace{7pt}

\noindent {\bf Proposition~E} (Proposition~\ref{prop:Drin-RHA-comod}){\bf .} {\it The reflective algebra $R_H(A)$ is a left $\Drin(H)$-comodule algebra, with an explicit comodule structure given in terms of the quantum $R$-matrix of $H$.} \qed

\vspace{7pt}

%%%%%%%%%%%%%%%%%%%%%%%%%%
%%%%%%%%%%%%%%%%%%%%%%%%%%
%%%%%%%%%%%%%%%%%%%%%%%%%%

\section{Preliminaries on (braided) monoidal categories}\label{sec:back-C}

In this section, we review terminology pertaining to braided monoidal categories. We refer the reader to  \cite{EGNO} and \cite{TV} for general information. We first review background material on monoidal categories  in Section~\ref{sec:monoidal}. Then in  Section~\ref{sec:braided}, we recall braided monoidal categories, the Drinfeld center construction of a braided category from a monoidal category, and the connection to quantum $R$-matrices in Hopf case. We assume that all categories here are {\it locally small} (i.e., the collection of morphisms between any two objects is a set).

%%%%%%%%%%%%%%%%%%%%%%%%%%
\subsection{Monoidal categories} \label{sec:monoidal}

We refer the reader to  \cite{EGNO, TV, Walton}
%\cite{EGNO}*{Sections~1.8, 2.1, 2.4, 2.10, 4.1, 4.2} and \cite{TV}*{Sections~1.1, 1.4, 1.6, 4.1, 4.2, 4.6}
for further details. 

\smallskip

{\bf Monoidal categories.} A {\it monoidal category} consists of a category $\cC$ equipped with a bifunctor $\otimes\colon  \cC \times \cC \to \cC$, a natural isomorphism $a_{X,Y,Z}\colon  (X \otimes Y) \otimes Z \overset{\sim}{\to} X \otimes (Y \otimes Z)$ for $X,Y,Z \in \cC$, an object $\one \in \cC$, and natural isomorphisms $l_X\colon  \one \otimes X \overset{\sim}{\to} X$ and $r_X\colon  X  \otimes \one \overset{\sim}{\to} X$ for  $X \in \cC$, satisfying pentagon and triangle axioms.

A {\it (strong) monoidal functor}  between monoidal categories $(\cC, \otimes, \one, a, l, r)$ and $(\cC', \otimes', \one', a', l', r')$ is a functor $F\colon  \cC \to \cC'$ equipped with a natural isomorphism $F_{X,Y}\colon  F(X) \otimes' F(Y) \isomorph F(X \otimes Y)$ for $X,Y \in \cC$, and an isomorphism $F_0\colon  \one' \isomorph F(\one)$ in $\cC'$, satisfying associativity and unitality constraints. \linebreak

\vspace{-.2in}

An {\it equivalence (resp., isomorphism) of  monoidal categories} is provided by a  monoidal functor between the two monoidal categories that yields an equivalence (resp., isomorphism) of the underlying categories; it is denoted by $\overset{\otimes}{\simeq}$ (resp., $\overset{\otimes}{\cong}$).

\smallskip

{\bf Opposite monoidal category.} Given a monoidal category $(\cC, \otimes, \one, a, l, r)$, its {\it opposite monoidal category} is defined as $\cC^{\otimes \textnormal{op}}:=(\cC, \otimes^{\textnormal{op}}, \one, a^{\textnormal{op}}, l^{\textnormal{op}}, r^{\textnormal{op}})$, with $X \otimes^{\textnormal{op}} Y := Y \otimes X$ and $a^{\textnormal{op}} _{X,Y,Z} := a^{-1}_{Z,Y,X}$ and $l_X^{\textnormal{op}} = r_X$  and $r_X^{\textnormal{op}} = l_X$, for all $X,Y,Z \in \cC$.

\smallskip

{\bf Rigidity.} A monoidal category $(\cC, \otimes, \one)$ is {\it rigid} if it comes equipped with left and right dual objects,  i.e., for each $X \in \cC$ there exist, respectively, an object $X^* \in \cC$ with co/evaluation maps $\ev_X^L \colon  X^* \otimes X \to \one$ and $\coev_X^L \colon  \one \to  X \otimes X^*$, and an object ${}^*X \in \cC$ with co/evaluation maps $\ev_X^R \colon  X \otimes {}^*X \to \one$, $\coev_X^R \colon \one \to  {}^*X \otimes X$,  satisfying  coherence conditions of left and right duals.

\smallskip

{\bf Linearity over $\Bbbk$, finiteness.} We now discuss certain $\Bbbk$-linear monoidal categories. A $\Bbbk$-linear abelian category $\cC$ is {\it locally finite} if, for any two objects $V,W$ in $\cC$, $\Hom_{\cC}(V,W)$ is a finite-dimensional $\Bbbk$-vector space and every object has finite length. A locally finite category $\cC$ is {\it finite} if there are enough projectives and finitely many isomorphism classes of simple objects. Equivalently, a $\Bbbk$-linear category $\cC$ is  finite if it is equivalent to the category of finite-dimensional modules over a finite-dimensional $\Bbbk$-algebra.

\smallskip

{\bf Tensor and fusion categories.} A  {\it tensor category} is an abelian, $\Bbbk$-linear, locally finite, rigid, monoidal category $(\cC, \otimes, \one)$ such that $\otimes$ is $\Bbbk$-linear in each slot and $\End_\cC(\one) \cong \Bbbk$.
 A {\it tensor functor} is a $\Bbbk$-linear, exact, faithful,  monoidal functor $F$ between  tensor categories $\cC$ and $\cC'$, with $F(\one) = \one'$. A  tensor category is said to be a {\it fusion category} if it is both finite and semisimple. 
 If $\cC$ is a tensor (resp., finite tensor, fusion) category, then is so $\cC^{\otimes \textnormal{op}}$.
 
\smallskip

{\bf Deligne tensor product.} Let $\cC$ and $\cC'$ be two tensor categories. Then the {\it Deligne tensor product} of $\cC$ and $\cC'$ is a $\Bbbk$-linear, abelian category $\cC \boxtimes \cC'$ endowed with a functor $\boxtimes: \cC \times \cC' \to \cC \boxtimes \cC'$ that is $\Bbbk$-linear and right exact in each variable, and is universal among such functors out of $\cC \times \cC$.  We also have that $\cC \boxtimes \cC'$ is a monoidal  category where $(X \boxtimes X') \otimes^{\cC \boxtimes \cC'} (Y \boxtimes Y') := (X \otimes^{\cC} Y) \boxtimes (X' \otimes^{\cC'} Y')$, for $X, Y \in \cC$, $X', Y' \in \cC'$, and $\one^{\cC \boxtimes \cC'} := \one^{\cC} \boxtimes \one^{\cC'}$.  Moreover, the Deligne tensor product of two  tensor (resp., finite tensor, fusion) categories is a tensor (resp.,  finite tensor, fusion) category.

\smallskip

{\bf Hopf case.}  The category $H$-$\mathsf{fdmod}$ of finite-dimensional $\Bbbk$-modules over a (finite-dimensional) Hopf algebra $H$ is a (finite) tensor category. If, further, $H$ is a semisimple Hopf algebra, then $H$-$\mathsf{fdmod}$ is a fusion category.
If $H$ and $H'$ are Hopf algebras over $\Bbbk$, then $(H\text{-}\mathsf{mod})^{\otimes \textnormal{op}} \overset{\otimes}{\simeq} H^{\textnormal{cop}}\text{-}\mathsf{mod}$, for the coopposite Hopf algebra $H^{\textnormal{cop}}$. We also have that $H\text{-}\mathsf{mod} \boxtimes H'\text{-}\mathsf{mod} \overset{\otimes}{\simeq} (H \otimes_\Bbbk H')\text{-}\mathsf{mod}$, for the standard tensor product of Hopf algebras $H \otimes_\Bbbk H'$ over $\Bbbk$. 
%%%%%%%%%%%%%%%%%%%%%%%%%%

\subsection{Braided categories, Drinfeld centers, and quantum $R$-matrices}
\label{sec:braided}
See \cite{EGNO}*{Sections~8.1, 8.3, 8.5, 8.6}, \cite{Majid}*{Sections 2.1, 7.1}, \cite{HeckenbergerSchneider}*{Section~4.1},  \cite{Kassel} for further details.  

\pagebreak

{\bf Braided categories.}  A  monoidal category  $(\cC, \otimes, \one, a, l, r)$ is {\it braided} if it is a equipped with  a natural isomorphism $c_{X,Y}\colon  X \otimes Y \overset{\sim}{\to} Y \otimes X$ for $X,Y \in \cC$\; ({\it braiding}),
such that the following hexagon axioms hold for each $X,Y,Z \in \cC$: 
\begin{eqnarray}
c_{X \otimes Y, Z} &=& a_{Z,X,Y} \circ (c_{X,Z} \otimes \ide_Y) \circ a_{X,Z,Y}^{-1} \circ (\ide_X \otimes c_{Y,Z})  \circ a_{X,Y,Z}, \label{eq:braid1} \\
c_{X, Y \otimes Z} &=&a_{Y,Z,X}^{-1} \circ  (\ide_Y \otimes c_{X,Z}) \circ a_{Y,X,Z} \circ (c_{X,Y} \otimes \ide_Z)  \circ a_{X,Y,Z}^{-1}. \label{eq:braid2}
\end{eqnarray}

We also have a {\it mirror braiding} on $\cC$ given by 
$c_{Y,X}^{-1}\colon X \otimes Y \overset{\sim}{\to} Y \otimes X$
for $X,Y \in \cC$. 
We refer to the braided monoidal category
$\overline{\cC}:=(\cC, \otimes, \one, a,l,r, c^{-1})$
as the {\it mirror} of $(\cC, \otimes, \one, a,l,r,c)$.

\smallskip

A {\it braided monoidal functor} between braided monoidal categories $\cC$ and $\cC'$ is a monoidal functor $(F, F_{-,-},F_0)\colon  \cC \to \cC'$ such that 
\begin{equation} \label{eq:braidfunc}
F_{Y,X} \circ c'_{F(X),F(Y)} = F(c_{X,Y}) \circ  F_{X,Y}
\end{equation}
for all $X,Y \in \cC$. An {\it equivalence (resp., isomorphism) of  braided monoidal categories} is a braided monoidal functor that yields an equivalence (resp., isomorphism) of the underlying categories. Similar notions exist for tensor categories and tensor functors.

\smallskip

{\bf Drinfeld centers.}  An important example of a braided monoidal category is the {\it Drinfeld center}  $\cZ(\cC)$ of a monoidal category $(\cC, \otimes, \one)$. Its objects are pairs $(V, c^V)$, where $V$ is an object of $\cC$ and $$c^V:=\{c^V_X \colon  X \otimes V \overset{\sim}{\to} V \otimes X\}_{X \in \cC}$$ is a natural isomorphism, called a  {\it half-braiding}, satisfying 
\begin{eqnarray} \label{eq:half-braid}
c^V_{X \otimes Y} &= \; a_{V,X,Y} \circ (c^V_X \otimes \ide_{Y}) \circ a^{-1}_{X,V,Y} \circ (\ide_{X} \otimes c^V_Y) \circ a_{X,Y,V}.
\end{eqnarray}
Morphisms $(V, c^V) \to (W, c^W)$ of $\cZ(\cC)$ are given by
  $f \in \Hom_{\cC}(V,W)$ such that, for all $X \in \cC$:
 \[
 (f \otimes \ide_X) \circ c^V_X = c^W_X \circ (\ide_X \otimes f).
 \]

The monoidal product of $\cZ(\cC)$ is $(V,c^V) \otimes (W,c^W):=(V\otimes W, c^{V \otimes W})$,  for all $X \in \cC$: 
\begin{eqnarray} \label{eq:center-tens}
c^{V\otimes W}_X \; := \; a_{V,W,X}^{-1} \circ (\ide_V \otimes c^W_X) \circ a_{V,X,W} \circ (c^V_X \otimes \ide_W) \circ a^{-1}_{X,V,W}.
\end{eqnarray}
An important feature of $\cZ(\cC)$ is the braiding defined by 

\vspace{-.2in}

\[
\xymatrix@C=4pc{
c_{V,W}\; :=c_{(V,c^V),(W,c^W)}:  (V \otimes W,c^{V \otimes W}) \ar[r]^(.65){c^W_V}
&(W \otimes V,c^{W \otimes V}).
}
\]
Moreover, if $\cC$ is a (finite) tensor category, then $\cZ(\cC)$ is a braided (finite) tensor category. 

\smallskip

%{\bf Embeddings into Drinfeld centers in the braided case.}
% Given a braided monoidal category $(\cC,c)$, let $\overline{\cC}:=(\cC,c^{-1})$ be the mirror braided monoidal category. Then we have the following embeddings of $\cC$ into its Drinfeld center as braided monoidal categories:
%\[
%I_+: \cC \to \cZ(\cC), \; V \mapsto (V, c^V) \quad \quad \quad I_-: \bar{\cC} \to \cZ(\cC), \; V \mapsto (V, (c^{-1})^V).
%\] 

{\bf Drinfeld centers in Hopf case and Yetter--Drinfeld modules.}   Take a finite-dimensional Hopf algebra $H:=(H,m,u,\Delta,\varepsilon,S)$ over $\Bbbk$, where $\Delta(h) =: h_{(1)} \otimes_\Bbbk h_{(2)}$ (sumless Sweedler notation). When $\cC = H$-$\mathsf{mod}$, we have the isomorphisms of braided monoidal categories below, 
\begin{equation} \label{eq:Drin}
   \cZ(H\text{-}\mathsf{mod}) \;\; \overset{\otimes}{\cong}  \;\; \lYD{H}  \;\; \overset{\otimes}{\cong}  \; \;\Drin(H)\text{-}\mathsf{mod},
\end{equation}
where $\lYD{H}$ is the category of  left Yetter--Drinfeld modules over $H$, and $\Drin(H)$ is the  Drinfeld double of $H$. We provide the details below.

The objects of the category of {\it left Yetter--Drinfeld modules} $\lYD{H}$ are triples $(V, \varodot, \partial^V)$, where $(V, \varodot)$ is a left $H$-module, and $(V, \partial^V)$ is a left $H$-comodule with $\partial^V(v):= v_{\langle -1 \rangle} \otimes_\Bbbk v_{\langle 0 \rangle} \in H \otimes_\Bbbk V$, subject to the following compatibility condition between $\varodot$ and $\partial^V$:
$$
\partial^V(h \varodot v) = h_{(1)} v_{\langle -1 \rangle} S(h_{(3)}) \otimes_\Bbbk (h_{(2)} \varodot v_{\langle 0 \rangle}).
$$
A morphism of $\lYD{H}$ is a linear map which is simultaneously a left $H$-module morphism and a left $H$-comodule morphism. Then the first isomorphism of \eqref{eq:Drin} holds via the assignments:
\begin{equation} \label{eq:ZC-YD}
{\small
\begin{array}{rl}
\cZ(\lmod{H}) \ni (V, \varodot, c^V) & \hspace{-.05in} \mapsto (V, \varodot, \partial^V) \in \lYD{H} \text{ for } \partial^V(v) := (c^V_H)^{-1}(v \otimes_\Bbbk 1_H) \\[.4pc]
\cZ(\lmod{H}) \ni (V, \varodot, c^V)  & \hspace{-.05in} \mapsfrom  (V, \varodot, \partial^V)  \in \lYD{H}\\[.2pc]
\text{ for } c^V_X(x \otimes_\Bbbk v):= v_{\langle 0 \rangle} \otimes_\Bbbk ((S^{-1}(v_{\langle -1 \rangle}) \cdot x). 
\end{array}
}
\end{equation}
Here,  $H$ in $c^V_H$ is the regular left $H$-module, and $(X,\cdot)$ is an arbitrary left $H$-module.

On the other hand, the {\it Drinfeld double} of $H$ is a Hopf algebra $\Drin(H)$, which is, as a start, equal to  $H^* \otimes_\Bbbk H$ as a vector space. Next, denote the standard  left and right actions of $H$ on $H^*$ by $\twoheadrightarrow$ and $\twoheadleftarrow$, respectively. That is, $h \twoheadrightarrow \xi := \langle \xi_{(2)}, h \rangle \xi_{(1)}$ with $\langle h \twoheadrightarrow \xi, h' \rangle = \langle \xi, h'h \rangle$,  and $\xi \twoheadleftarrow h := \langle \xi_{(1)}, h \rangle \xi_{(2)}$ with $\langle \xi \twoheadleftarrow h, h' \rangle = \langle \xi, h h' \rangle$, with $h, h' \in H$ and $\xi \in H^*$. Here, we use the Hopf pairing between $H^*$ and $H$. Then $\Drin(H)$ contains $H$ and $(H^*)^{\textnormal{op}}$ as Hopf subalgebras, and the product of  $\Drin(H)$ between elements of $H$ and $(H^*)^{\textnormal{op}}$ is given by: 
\begin{equation} \label{eq:Drin-prod}
h \xi \; =\; (h_{(3)} \twoheadrightarrow \xi \twoheadleftarrow S(h_{(1)})) h_{(2)} =  \langle \xi_{(1)}, S(h_{(1)}) \rangle \; \langle \xi_{(3)}, h_{(3)} \rangle\;  \xi_{(2)} h_{(2)},
\end{equation}
for $h \in H, \; \xi \in (H^*)^{\textnormal{op}}$.
Moreover, $\Drin(H)$ has the tensor product unit, coproduct, and counit. Then the second isomorphism of \eqref{eq:Drin} holds via  the assignments:
\begin{equation} \label{eq:YD-Drin}
{\small
\begin{array}{rl}
\lYD{H} \ni  (V, \varodot, \partial^V) & \hspace{-.05in} \mapsto  (V, \varodot, \diam) \in \lmod{\Drin(H)} \text{ for } \xi \diam v := \langle \xi, v_{\langle -1 \rangle} \rangle v_{\langle 0 \rangle}\\[.4pc]
\lYD{H} \ni (V, \varodot, \partial^V) & \hspace{-.05in} \mapsfrom  (V, \varodot, \diam)  \in \lmod{\Drin(H)} \\[.2pc]
\; \;  \text{ for } \partial^V(v):= \textstyle \sum_d h_d \otimes_\Bbbk (\xi_d \diam v),
\end{array}
}
\end{equation} 
with $\{h_d, \xi_d\}_d$ a dual basis of $H$. (The latter assignment is independent of choice of dual basis of~$H$.) \linebreak

\vspace{-.13in}
%\medskip

{\bf Quasitriangular Hopf algebras, quantum $R$-matrices.}
Again, take a Hopf algebra $H:=(H,m,u,\Delta,\varepsilon,S)$ over $\Bbbk$, where $\Delta(h) =: h_{(1)} \otimes_\Bbbk h_{(2)}$ (sumless Sweedler notation). Moreover, denote $\otimes:= \otimes_\Bbbk$. We say that $H$ is {\it quasitriangular} if there exists an invertible element 
\begin{equation}
\label{eq:R-sum-notation}
R = \textstyle \sum_i s_i \otimes t_i \in H \otimes H \;\; \text{({\it quantum $R$-matrix})}, 
\end{equation}
with inverse, $R^{-1}:= \textstyle \sum_i s^i \otimes t^i \; \in H \otimes H$, such that 
\begin{align}
\label{eq:QT1} \textstyle \sum_i (s_i)_{(1)} \otimes (s_i)_{(2)} \otimes t_i
 \; &=  \textstyle\; \sum_{j,k} s_j \otimes s_k \otimes t_j t_k,\\[.2pc]
\label{eq:QT2}  \textstyle\sum_i s_i \otimes (t_i)_{(1)} \otimes (t_i)_{(2)}
 \;  &= \textstyle\;  \sum_{j,k} s_j s_k \otimes t_k \otimes t_j,\\[.2pc]
\label{eq:QT3}  \textstyle \sum_i s_i h_{(1)} \otimes t_i h_{(2)}
  \; &= \textstyle\;  \sum_i  h_{(2)} s_i \otimes h_{(1)} t_i.
\end{align}
Alternatively, we will also use the following notation for quantum $R$-matrices. Take 
\[
R_{ab}:= \textstyle \sum_i s_i \; \text{(in the $a$-th slot)}\; \otimes \; t_i \; \text{(in the $b$-th slot)}
\otimes \; 1_H \; \text{(in other slots)},
\]
e.g., $R_{13} = \sum_i s_i \otimes 1_H \otimes t_i$. Then, the conditions \eqref{eq:QT1}--\eqref{eq:QT3} are written, respectively, as follows:
\[
(\Delta \otimes \ide_H) (R) = R_{13} R_{23},  
\quad \quad
(\ide_H \otimes \Delta) (R) = R_{13} R_{12},
\quad \quad
R \Delta(h) = \Delta^{\mathrm{op}}(h) R \quad \forall h \in H.  
\]

 \pagebreak
 
For a quasitriangular Hopf algebra $(H,R)$, we also have the identities below:
\begin{align}
\label{eq:QT4}
&(\varepsilon \otimes \ide)(R) = 1_H, \hspace{.28in}
(\ide \otimes \varepsilon)(R) = 1_H, \\[.2pc]
\label{eq:QT5}
&R^{-1} = (S \otimes \ide)(R), \quad
R = (\ide \otimes S)(R^{-1}), \quad
R = (S \otimes S)(R),
\end{align}
\smallskip
where $\Bbbk \otimes H$ and $H \otimes \Bbbk$ are identified with $H$.

\smallskip

For example, for a finite-dimensional Hopf algebra $H$ with dual bases $\{h_d, \xi_d\}_d$, we have that the Drinfeld double $\Drin(H)$ of $H$ is quasitriangular, 
with $R$-matrix:
\begin{equation*}
R_{\Drin(H)} := \textstyle \sum_d 1_{H^*} \otimes h_d \otimes \xi_d \otimes 1_H \; \in \Drin(H) \otimes \Drin(H). 
\end{equation*}

Moreover, quantum $R$-matrices of $H$ are tied to braidings of $\lmod{H}$ as we see  below.

\begin{lemma} \label{lem:R-c}
The tensor category $\lmod{H}$ is braided with 
\[
c_{X,Y} : X \otimes Y \to Y \otimes X, \quad
x \otimes y \mapsto \textstyle \sum_i (t_i \cdot y) \otimes (s_i \cdot x),
\]
for $R:= \sum_i s_i \otimes t_i \in H \otimes H$, if and only if  $R$ is a quantum $R$-matrix for $H$. \qed
\end{lemma}

%%%%%%%%%%%%%%%%%%%%%%%%%%
%%%%%%%%%%%%%%%%%%%%%%%%%%
%%%%%%%%%%%%%%%%%%%%%%%%%%

\section{Preliminaries and results on (braided) module categories} \label{sec:back-M}

Throughout this section, let $(\cC, \otimes, \one, a, l, r)$ be a monoidal category, unless stated otherwise. We review module categories over $\cC$ and module functors between them in Section~\ref{sec:modcat}. The collection of module functors form a category, which we discuss in Section~\ref{sec:catmodfunc}. Exact module categories are reviewed in Section~\ref{sec:exmodcat}. Then, bimodule categories and their centers are discussed in Section~\ref{sec:bimod-center}. Finally, braided module categories are introduced in Section~\ref{sec:bmc}, braided module functors are studied in Section~\ref{sec:bmcfunc}, and connections to quantum $K$-matrices are presented in Section~\ref{sec:Kmatrix}.

%%%%%%%%%%%%%%%%%%%%%%%%%%
\subsection{Module categories} \label{sec:modcat}
See \cite{EGNO}*{Sections~7.1--7.3} for details. A {\it left $\cC$-module category}  is a category $\cM$ equipped with
an {\it action bifunctor} $\act: \cC \times \cM \to \cM$, 
a natural isomorphism $m_{X,Y,M}\colon  (X \otimes Y) \act M \overset{\sim}{\to} X \act (Y \act M)$ for $X,Y \in \cC$, $M \in \cM$, and a natural isomorphism $\lambda_M\colon  \one \act M \overset{\sim}{\to} M$ for $M \in \cM$,
such that the following pentagon and triangle axioms hold for each $X,Y,Z \in \cC$ and $M \in \cM$:
\begin{eqnarray}
m_{X,Y,Z \act M} \circ m_{X \otimes Y, Z,M} &=& (\ide_X \act m_{Y,Z,M})\circ m_{X, Y \otimes Z, M}\circ (a_{X,Y,Z} \act \ide_M), \label{eq:modassoc} \\
r_X \act \ide_M &=& (\ide_X \act \lambda_M) \circ m_{X, \one, M}.  \label{eq:modunit}
\end{eqnarray}
We sometimes write $\cM$ or $(\cM, \act)$ to denote $(\cM, \act, m, \lambda)$ for brevity.

\smallskip

A {\it $\cC$-module functor} between left $\cC$-module categories $(\cM, \act, m, \lambda)$ and $(\cM', \act', m', \lambda')$ is a functor $F: \cM \to \cM'$ equipped with a natural isomorphism,
$$s:=\{s_{X,M}: F(X \act M) \isomorph X \act' F(M)\}_{X \in \cC, M \in \cM},$$ 
such that the following coherence axioms hold for each $X,Y \in \cC$ and $M \in \cM$:
\begin{eqnarray}
m'_{X,Y,F(M)} \circ s_{X \otimes Y, M} &=& (\ide_X \act' s_{Y,M})\circ s_{X, Y \act M}\circ F(m_{X,Y,M}), \label{eq:modfuncassoc}\\
F(\lambda_M) &=& \lambda'_{F(M)} \circ s_{\one, M}. \label{eq:modfuncunit}
\end{eqnarray}

\pagebreak 

Similarly, a  \emph{right $\cC$-module category} is a category $\cM$ equipped with  a bifunctor $\triangleleft\colon  \cM \times \cC \to \cM$, a natural isomorphism $n_{M,X,Y}\colon  M \triangleleft (X \otimes Y) \overset{\sim}{\to} (M \triangleleft X) \triangleleft Y$ for $X,Y \in \cC$, $M \in \cM$, and a natural isomorphism $\rho_M\colon M\triangleleft \one  \overset{\sim}{\to} M$ for $M \in \cM$, satisfying pentagon and triangle axioms.

A {\it $\cC$-module functor} between right $\cC$-module categories $(\cM, \ract, n, \rho)$ and $(\cM', \ract', n', \rho')$ is a functor $F: \cM \to \cM'$ equipped with a natural isomorphism,
$$t:=\{t_{M,X}: F(M \ract X) \isomorph F(M) \ract' X\}_{X \in \cC, M \in \cM},$$ 
satisfying coherence axioms.

A {\it left module category over a tensor category $\cC$} is a left $\cC$-module category $(\cM, \act)$ that is abelian, $\Bbbk$-linear, locally finite, bilinear on morphisms, such that $- \act M: \cC \to \cM$ is exact for all $M \in \cM$. A similar notion holds for right module categories.
We also assume that module functors between such module categories are additive in each slot.

%%%%%%%%%%%%%%%%%%%%%%%%%%

\subsection{Categories of module functors} \label{sec:catmodfunc}

The collection of $\cC$-module functors between left $\cC$-module categories $\cM$ and $\cM'$ forms a category, which we denote by $\mathsf{Fun}_{\cC}(\cM,\cM')$. A morphism in $\mathsf{Fun}_{\cC}(\cM,\cM')$ from $(F_1, s_1)$ to $(F_2,s_2)$ is a natural transformation from $F_1$ to $F_2$ that is compatible with $s_1$ and $s_2$. See \cite{EGNO}*{Section~7.2} for more information.

\smallskip

The category $\mathsf{Fun}_{\cC}(\cM,\cM')$ is not always well-behaved, but we have the following useful result.
 
\begin{proposition} \cite{ENO}*{Theorem~2.16} \label{prop:FunCss}
 If $\cC$ is a fusion category, and $\cM$ and $\cM'$ are finite and semisimple left $\cC$-module categories, then $\mathsf{Fun}_{\cC}(\cM, \cM')$ is a semisimple category. \qed
 \end{proposition}

A  subcollection of $\mathsf{Fun}_{\cC}(\cM,\cM')$ that is better behaved is the collection of right exact $\cC$-module functors between left $\cC$-module categories $\cM$ and $\cM'$; their  full subcategory of $\mathsf{Fun}_{\cC}(\cM,\cM')$  is denoted by $\mathsf{Rex}_{\cC}(\cM,\cM')$. 
We also denote the full subcategory of exact $\cC$-module functors from $\cM$ to $\cM'$ by $\mathsf{Ex}_{\cC}(\cM,\cM')$.

%%%%%%%%%%%%%%%%%%%%%%%%%%

\subsection{Exact module categories} \label{sec:exmodcat}
Assume here that $\cC$ is a  finite tensor category. Here, we recall background material on exact module categories from \cite{EGNO}*{Sections~7.5, 7.11}.

A locally finite module category $(\cM, \act)$ over $\cC$ is called {\it exact} if for any projective object $P \in \cC$ and any object $M \in \cM$ we have that the object $P \act M$ is projective in $\cM$.

\begin{example} \label{ex:Cexact}
If $\cC$ is a finite tensor category, then $\cC$ is an exact left module category over both $\cC$ and  $\cC \boxtimes \cC^{\otimes \textnormal{op}}$; see \cite{EO}*{Example~3.3(i)}.
\end{example}

As mentioned above, the category of right exact module category functors is better behaved  than the category of ordinary module category functors. To see this, consider the result below.

\begin{proposition} \cite{EGNO}*{Section~7.11}  \label{prop:FunCabelfin}
 If $\cM$ and $\cM'$ are exact, finite, left $\cC$-module categories, then 
 $\mathsf{Rex}_{\cC}(\cM, \cM')$ is abelian and finite.
\qed
\end{proposition}

We also have that any module category functor from an exact module category is (right) exact.

\begin{proposition} \cite{EO}*{Proposition~3.11}  \label{prop:FunCRexC} Let $\cM$ be an exact, finite left $\cC$-module category.
Then $\mathsf{Fun}_{\cC}(\cM, \cM') = \mathsf{Ex}_{\cC}(\cM, \cM')$ for any left $\cC$-module category $\cM'$. \qed
\end{proposition}

%%%%%%%%%%%%%%%%%%%%%%%%%%

\subsection{Bimodule categories and their centers} \label{sec:bimod-center}
Here, we recall material from work of Greenough, \cite{Greenough}*{Sections~2 and~7}.
A \emph{$\cC$-bimodule category} is a tuple $(\cM,\act, \triangleleft, m,n,\lambda,\rho)$ such that $(\cM,\act, m,\lambda)$ is a left  $\cC$-module category, and $(\cM, \triangleleft, n,\rho)$ is a right $\cC$-module category, with a natural isomorphism,
$b:=\{b_{X,M,Y}\colon (X\act M)\triangleleft Y \isomorph X\act (M\triangleleft Y)\}_{X,Y\in \cC, M\in \cM},$
satisfying  compatibility conditions. 

\begin{remark} \label{rem:bimod}
Note that $(\cM,\act, \triangleleft, m,n,\lambda,\rho)$ is a $\cC$-bimodule category if and only if $(\cM,\bar{\act}, \bar{m}, \bar{\lambda})$ a left module category over $\cC \boxtimes \cC^{\otimes \textnormal{op}}$. Here, $(X \boxtimes X') \bar{\act} M \leftrightsquigarrow (X \act M) \triangleleft X'$, for $X,X' \in \cC$ and $M \in \cM$, and we have similar correspondences between the associativity and unitality constraints.
\end{remark}

A {\it functor of $\cC$-bimodule categories} $F\colon \cM\to \cM'$ is at the same time a functor for the left and right $\cC$-module structures, with natural isomorphisms for $X\in \cC$ and $M\in \cM$,
$$s_{X,M}\colon F(X\act M)\isomorph X\act' F(M), \qquad t_{M,X}\colon F(M\triangleleft X)\isomorph F(M)\triangleleft' X,$$
satisfying the compatibility condition below for all $X,Y\in\cC$, $M\in \cM$:
\begin{gather*}
 (\ide_X\act' t_{M,Y}) \circ  s_{X,M\triangleleft Y} \circ F(b_{X,M,Y})  = b'_{X,F(M),Y}  \circ (s_{X,M}\triangleleft' \ide_Y) \circ t_{X\act M,Y}.
\end{gather*}

Given a $\cC$-module category, one defines its center by analogy with the center of a monoidal category (cf. \Cref{sec:braided}).

\begin{definition} \label{def:ZCM} Given a $\cC$-bimodule category $\cM$, we define its \emph{center} $\cZ_{\cC}(\cM)$ as the category consisting of objects $(M,d^M)$, where $M$ is an object of $\cM$ and $d$ is a natural isomorphism,
$$d^M=\left\{d^M_X\colon X\act M\isomorph M\triangleleft X\right\}_{X\in \cC} \; \; \text{({\it half-braiding})},$$
which satisfies the  coherence condition below for all $X,Y \in \cC$ and $M \in \cM$. 
\begin{align} \label{eq:dXYM}
d_{X\otimes Y}^M = n_{M,X,Y}^{-1} \circ (d_{X}^{M}\triangleleft\ide_Y) \circ b^{-1}_{X,M,Y} \circ (\ide_X\act d_{Y}^{M}) \circ m_{X,Y,M}.
\end{align}

Morphisms $f\colon (M,d^M)\to (N,d^N)$ are given by morphisms $f\colon M\to N$ in $\cM$ which commute with the respective half-braidings, that is, for $X \in \cC$:
$
(f\triangleleft \ide_X) \circ d^M_X = d^N_X \circ (\ide_X\act f).
$
\end{definition}

\begin{example} \label{ex:ZC=ZCC}
The Drinfeld center $\cZ(\cC)$ of  $\cC$ appears as the special case $\cZ_\cC(\cC)$, where $\cC$ is a $\cC$-bimodule via the regular action $X\act V=X\otimes V$ and $V\triangleleft X = V \otimes X$, along with $\cC$-bimodule constraints derived from the monoidal constraints of $\cC$ as follows: $m_{X,Y,V} = a_{X,Y,V}$, $n_{V,X,Y} = a^{-1}_{V,X,Y}$, $\lambda_X = l_X$, $\rho_X = r_X$, and $b_{X,V,Y} = a_{X,V,Y}$ for all $X,Y,V \in \cC$. In this case, $d_X^V = c_X^V$  for all $X,V \in \cC$.
\end{example}

%%%%%%%%%%%%%%%%%%%%%%%%%%

\subsection{Braided module categories} \label{sec:bmc} Now assume that $\cC:=(\cC,c)$ is braided. We say that a left $\cC$-module category $(\cM, \act, m, \lambda)$ is {\it braided} if it is equipped with a natural isomorphism, 
$$e:=\{e_{X,M}\colon  X \act M \overset{\sim}{\to}  X \act M\}_{X \in \cC, M \in \cM} \; \; \text{({\it braiding})},$$ 
 such that the following axioms hold for each $X,Y \in \cC$ and $M \in \cM$:
\begin{align}
 \begin{split}
 e_{X \otimes Y, M}
 \; =\;  \;  &  m_{X,Y,M}^{-1} \circ (\ide_X \act e_{Y,M}) \circ\; m_{X,Y,M} \circ (c_{Y,X} \act \ide_M)
 \smallskip \\
 &\circ \; m_{Y,X,M}^{-1} \circ (\ide_Y \act e_{X,M}) \circ m_{Y,X,M}\circ (c_{Y,X}^{-1} \act \ide_M),     
 \end{split} \label{eq:brmod2}
 \\[.4pc]
\begin{split}
 e_{X,Y \act M}
 \; =\;  \;  & m_{X,Y,M} \circ (c_{Y,X} \act \ide_M) \circ m_{Y,X,M}^{-1} \circ (\ide_Y \act e_{X,M}) \smallskip \\
 &\circ\; m_{Y,X,M} \circ (c_{X,Y} \act \ide_M) \circ m_{X,Y,M}^{-1}.
 \end{split} \label{eq:brmod1}
 \end{align}
%\cw{[from Notes-CW-20220808 page 1]}

\begin{remark}
\label{rem:equiv-e}
Let us compare the definition above to the definition of a braided module 
category in other parts of the literature. 
\begin{enumerate}[(a)]
    \item In \cite{Kolb20}*{Section~2} the author uses right $\cC$-module categories and works in the strict case, and our axiom \eqref{eq:brmod1} is the braided module category axiom \cite{Kolb20}*{(2.3)}. Moreover, our axiom \eqref{eq:brmod2} is the equivalent to the axiom \cite{Kolb20}*{(2.4)} via a similar argument to \cite{Kolb20}*{Remarks~2.2 and 2.4} when $\cC$ is {\it ribbon}. 
    More precisely, if $\cC$ is a ribbon category with a twist transformation $\{\theta_X : X \overset{\sim}{\to} X\}_{X \in \cC}$, see \cite{EGNO}*{Definition 8.10.1}, then the natural isomorphism $e$ satisfies conditions \eqref{eq:brmod2}--\eqref{eq:brmod1} if and only if the natural isomorphism
    \[
    \widetilde{e}:=\{\widetilde{e}_{X,M}= e_{X,M} (\theta_X \act \ide_M): X \act M \overset{\sim}{\to} X \act M\}_{X \in \cC, M \in \cM}
    \]
    satisfies the left-hand versions of \cite{Kolb20}*{(2.3)--(2.4)}: \eqref{eq:brmod1} with $e$ replaced by $\widetilde{e}$, with 
{\small
\begin{align*}
 \hspace{.45in} \widetilde{e}_{X \otimes Y, M}
  =  \;  m_{X,Y,M} \; (c_{Y,X} \act \ide_M)
 \;  m_{Y,X,M}^{-1} \; (\ide_Y \act \widetilde{e}_{X,M}) \; m_{Y,X,M}\; (c_{X,Y} \act \ide_M) \; m_{X,Y,M}^{-1} \; (\ide_X \act \widetilde{e}_{Y,M}).     
 \end{align*}
 }
 %my-notes-Pf-Remark3.19
    We further note that the conventions of \cite{Kolb20} are identical to those in \cite{Brochier13}*{Section~5.1}.
    
    \smallskip
    
    \item Considering \cite{DavydovNikshych}*{Definition~4.1}, our axiom \eqref{eq:brmod1} is the same as their first braided module category axiom. Also, our axiom \eqref{eq:brmod2} is the equivalent to their second braided module category axiom, by using the first braided module axiom.
\end{enumerate}
% See Notes-CW-20220808 page 2 
\end{remark}

The next result shows that braided module categories can be obtained using braided monoidal functors, cf. \cite{DavydovNikshych}*{Proposition 4.12}.

\begin{proposition} \label{prop:brmodsource}
We have the following statements.
\begin{enumerate}[\upshape (a)]
\item Suppose that $\cC$ and $\cC'$ are monoidal categories and $(F, F_{-,-},F_0): \cC \to \cC'$ is a (strong) monoidal functor. If $(\cM',\act',m',\lambda')$ is a left $\cC'$-module category,  then $(\cM', \act, m, \lambda)$ is a left $\cC$-module category with $$X \act M := F(X) \act' M,$$ 
 $m_{X,Y,M} :=  m'_{F(X), F(Y), M} (F^{-1}_{X,Y} \act' \ide_M)$, and 
$\lambda_M :=  \lambda'_M  (F_0^{-1} \act' \ide_M)$, for $X,Y \in \cC$,  $M \in \cM'$.

\smallskip

\item If $(\cC,c)$ and $(\cC', c')$ are  braided monoidal categories and  $(F, F_{-,-},F_0): \cC \to \cC'$ is a braided monoidal functor, then the left $\cC$-module category $\cC'$ from part (a)  is braided with
$$e_{X,M} := c'_{M,F(X)} \circ c'_{F(X),M},$$
for all $X \in \cC$ and $M \in \cC'$.
\end{enumerate}
\end{proposition}

\begin{proof} (a)  See, e.g., \cite[Example~3.18]{Walton}.
%%%%%%%%%%%%%%%%%%%%%%%%
% Details in MonoidalFunctors-ModuleCategories-20230703-CW
%%%%%%%%%%%%%%%%%%%%%%%%

\smallskip

(b)  It suffices to establish \eqref{eq:brmod2} and \eqref{eq:brmod1} when $e_{X,M} := c'_{M,F(X)} \circ c'_{F(X),M}$, for $X \in \cC$ and $M \in  \cC'$. The following computation verifies \eqref{eq:brmod1}:
{\small
\begin{align*}
&m_{X,Y,M} \; (c_{Y,X} \act \ide_M) \; m_{Y,X,M}^{-1} \; (\ide_Y \act e_{X,M}) \; m_{Y,X,M} \; (c_{X,Y} \act \ide_M) \; m_{X,Y,M}^{-1}
\\[6pt]
&=  a'_{F(X), F(Y), M} \; (F^{-1}_{X,Y} \otimes' \ide_M) \; (F(c_{Y,X}) \otimes' \ide_M) \; (F_{Y,X} \otimes' \ide_M) \; a'^{-1}_{F(Y), F(X), M} \\
&\quad \circ (\ide_{F(Y)} \otimes'  c'_{M,F(X)}) \;  (\ide_{F(Y)} \otimes'  c'_{F(X),M})\\
& \quad \circ a'_{F(Y), F(X), M} \; (F^{-1}_{Y,X} \otimes' \ide_M) \; (F(c_{X,Y}) \otimes' \ide_M) \;  (F_{X,Y} \otimes' \ide_M) \; a'^{-1}_{F(X), F(Y), M} \\[6pt]
&=  a'_{F(X), F(Y), M} \; (c'_{F(Y),F(X)} \otimes' \ide_M)  \; a'^{-1}_{F(Y), F(X), M}  \; (\ide_{F(Y)} \otimes'  c'_{M,F(X)}) \;  (\ide_{F(Y)} \otimes'  c'_{F(X),M})\\
& \quad \circ a'_{F(Y), F(X), M}  \; (c'_{F(X),F(Y)} \otimes' \ide_M)  \; a'^{-1}_{F(X), F(Y), M} \\[6pt]
&=  a'_{F(X), F(Y), M} \; (c'_{F(Y),F(X)} \otimes' \ide_M) \; a'^{-1}_{F(Y), F(X), M} \; (\ide_{F(Y)} \otimes'  c'_{M,F(X)}) \; a'_{F(Y),M,F(X)}\\
&\quad \circ a'^{-1}_{F(Y),M,F(X)} \; (\ide_{F(Y)} \otimes'  c'_{F(X),M}) \; a'_{F(Y), F(X), M}  \; (c'_{F(X),F(Y)} \otimes' \ide_M)  \; a'^{-1}_{F(X), F(Y), M} \\[4pt]
& = c'_{F(Y)\otimes' M, F(X)} \; c'_{F(X), F(Y) \otimes' M} \\[4pt]
& = e_{X, Y \act M}.
\end{align*}
}

\noindent The first and last equations holds by definition; the second equation holds by the braided monoidal functor axiom \eqref{eq:braidfunc}; the third equation holds trivially; and the fourth equation holds by the braided monoidal category axioms \eqref{eq:braid1} and \eqref{eq:braid2}.

Likewise,  \eqref{eq:brmod2} holds by applying a combination of the braided monoidal functor axiom \eqref{eq:braidfunc} and the braided monoidal category axioms \eqref{eq:braid1} and \eqref{eq:braid2}.
% See Notes-CW-20220731, page 4
\end{proof}

%%%%%%%%%%%%%%%%%%%%%%%%%%

\subsection{Braided module functors} \label{sec:bmcfunc}

Here, we compare  braided module categories via the notions below. For ease, given a left $\cC$-module category $\cM$ with objects $X \in \cC$, $M \in \cM$, and morphisms $\psi \in \cC$, $\phi \in \cM$, we write $X \act \phi$ and $\psi \act \ide_M$ for the morphisms $\ide_X \act \phi$ and $\psi \act M$ in $\cM$, respectively.

\begin{definition} \label{def:brmodcat-equiv}
A {\it braided $\cC$-module functor} between braided left $\cC$-module categories $(\cM,\act, e)$ and $(\cM',\act', e')$ is a left $\cC$-module functor $(F, s)\colon  (\cM, \act) \to (\cM', \act')$ such that 
\begin{equation} \label{eq:braidmodfunc}
e'_{X,F(M)} \circ s_{X,M} = s_{X,M} \circ  F(e_{X,M})
\end{equation}
for all $X \in \cC$ and $M \in \cM$. An {\it equivalence (resp., isomorphism) of  braided left $\cC$-module categories} is given by two braided module functors $F\colon \cM\to \cN$ and $G\colon \cN\to \cM$ between the two module categories that yields an equivalence (resp., isomorphism) of the underlying categories.
\end{definition}

The next result is straight-forward; the reader may refer to  the ArXiv version 1 of this article if they are interested in the proof.

\begin{proposition}\label{prop:inverseCbraided}
    Let $(F,s): (\cM,\act,e) \to (\cM',\act',e')$ be a functor of braided left $\cC$-module categories and let $G\colon \cM'\to \cM$ be a quasi-inverse  of $F$. Then there exists a natural isomorphism $s'$ making $(G,s')$ a functor of braided left $\cC$-module categories. \qed
\end{proposition}

The consequence below is now straight-forward to establish.

\begin{corollary}
Braided (left) $\cC$-module equivalence is an equivalence relation for braided (left) $\cC$-module categories. \qed
\end{corollary}

 We now discuss how to transfer structure for (braided) module categories, particularly across an equivalence of categories. The proof is available in the ArXiv version 1 of this article.

\begin{proposition}\label{prop:structuretransfer}
Let $F\colon \cM\to \cM'$ be a category equivalence with a quasi-inverse   $G\colon \cM'\to\cM$.
\begin{enumerate}[\upshape (a)]
\item  If $(\cM, \act)$ is a left $\cC$-module category, then we can define a left $\cC$-module category structure $\act'$ on $\cM'$ via 
$$X\triangleright' N \; := \; F(X\triangleright G(N)),$$
for each $X \in \cC$ and $N \in \cM'$, such that both $F$ and $G$ are left $\cC$-module functors.
    
\medskip
    
\item A braiding $e$ on the left $\cC$-module category $\cM$ induces a braiding $e'$ on $\cM'$ via
$$e'_{X,N} \; := \; F(e_{X,G(N)})\colon \; X\act'N\isomorph X\act'N,$$
for each $X \in \cC$ and $N \in \cM'$,
with the left $\cC$-module category structure from part (a) such that $F,G$ preserve the braiding.  \qed
\end{enumerate}
\end{proposition}

%%%%%%%%%%%%%%%%%%%%%%%%%%

\subsection{Quasitriangular comodule algebras and quantum $K$-matrices} \label{sec:Kmatrix}
Assume that $H$ is a quasitriangular Hopf algebra with a
quantum $R$-matrix $R := \sum_i s_i \otimes t_i \in H \otimes H$. Here, $\otimes:= \otimes_\Bbbk$. Let $A$ be a left 
$H$-comodule algebra with coaction
\[
\delta : A \to H \otimes A, \quad a \mapsto a_{[-1]} \otimes a_{[0]}. 
\]
\begin{definition}
\label{def:quastr-coalg} 
We say that $A$ is a {\emph{quasitriangular left $H$-comodule algebra}} 
if it is equipped with an invertible element, 
\[
K:= \textstyle \sum_i g_i \otimes p_i \; \in H \otimes A \; \qquad \; \text{({\it quantum $K$-matrix})},
\]
with inverse, $K^{-1}:= \textstyle \sum_i g^i \otimes p^i \; \in H \otimes A$,  such that
\begin{align}
\begin{split}
\textstyle \sum_i (g_i)_{(1)} \otimes (g_i)_{(2)}\otimes p_i \;&= \; \textstyle \sum_{j,k,l,m} t_k g_l t^m \; \otimes \; g_j s_k s^m \; \otimes \;p_j p_l, \\[.1pc]
(\text{equivalently,} \quad  (\Delta \otimes \ide_A) K &= K_{23}R_{21} K_{13} R_{21}^{-1}),
\end{split}
\label{eq:K2}
\\[.4pc]
\begin{split}
\textstyle \sum_i g_i \otimes (p_i)_{[-1]}\otimes (p_i)_{[0]} \;&= \; \textstyle \sum_{j,k,l} t_j g_k s_l \; \otimes \; s_j t_l \; \otimes \;p_k, \\[.1pc]
(\text{equivalently,} \quad
(\ide_H \otimes \delta) K &= R_{21} K_{13} R_{12}),
\end{split}
\label{eq:K3}\\[.4pc]
\begin{split}
\textstyle \sum_i g_i a_{[-1]} \otimes p_i a_{[0]} \;&= \; \textstyle \sum_i a_{[-1]} g_i  \otimes a_{[0]}  p_i \quad \forall a \in A, \\[.1pc]
(\text{equivalently,} \quad K \delta(a) &= \delta(a) K \quad \forall a \in A).
 \end{split}
\label{eq:K1}
\end{align}
\end{definition}

\noindent Here, $K_{ab}:= \textstyle \sum_i g_i \; \text{(in the $a$-th slot)}\; \otimes \; p_i \; \text{(in the $b$-th slot)}
\otimes \; 1_H \; \text{(in other slots)}.$

\medskip

Returning to braided module categories, consider the result below.

\begin{lemma} 
\label{lem:braid-Kmatr}
Retain the notation above for the quasitriangular Hopf algebra $(H,R)$, and the left $H$-comodule algebra $A$. Then, the following statements hold.
\begin{enumerate}[\upshape (a)]
\item We have that $A\text{-}\mathsf{mod}$ is a left module category over $\lmod{H}$ via
\[
\act: \lmod{H} \times A\text{-}\mathsf{mod} \to A\text{-}\mathsf{mod}, \quad \left((X, \cdot), (M, \ast)\right) \mapsto (X \otimes M,\;  \widetilde{\ast}),
\]
where $a \; \widetilde{\ast} \; (x \otimes m) := (a_{[-1]} \cdot x) \otimes (a_{[0]} \ast m)$ for $a \in A$, $x \in X$, $m \in M$.

\smallskip

\item Take $K:= \sum_i g_i \otimes p_i \in H \otimes A$, and for $(X, \cdot) \in \lmod{H}$,  $(M, \ast) \in A\text{-}\mathsf{mod}$, take the morphism:
\[
e_{X,M}: X \otimes M \to X \otimes M, \quad
x \otimes m \mapsto \textstyle \sum_i (g_i \cdot x) \otimes (p_i \ast m).
\]
Then, the left $(\lmod{H})$-module category $(A\text{-}\mathsf{mod}, \act)$ is braided with braiding given by
$e_{X,M}$ if and only if  $K$ is a quantum $K$-matrix for $A$.

\smallskip

\item 
Conversely, any braiding for the $(\lmod{H})$-module category $\lmod{A}$ from part (a) is of the form given in part (b) for some element $K:= \textstyle  \sum_i g_i \otimes p_i \in H\otimes A$.
\end{enumerate}
\end{lemma}

\begin{proof} Part (a) is straight-forward to check, and we leave this to the reader. For part (b), we sketch the forward direction; the reverse direction is proved by reversing the arguments. 

To proceed, note that if we write $H$ (resp.,  $A$) in the subscript of either the braiding $c$ or $e$, then this denotes the regular left $H$-module (resp.,  regular left $A$-module).
Also, denote the inverse of $e_{X,M}$ by $e_{X,M}^{-1} : X \otimes M \to X \otimes M$
and set
\[
K' := 
e_{H,A}^{-1} 
(1_H \otimes 1_A)  \in H \otimes A.
\]
By the naturality of $e_{X,M}^{-1}$, we obtain that 
$e_{X,M}^{-1}(x \otimes m) = K' \ \widetilde{\ast} \; (x \otimes m)$,
for $x \in X$, $m \in M$. This implies  $K K' = K' K = 1_H \otimes 1_A$.
Therefore, $K \in H \otimes A$ is invertible.

To establish \eqref{eq:K2}, we compute:
{\small
\begin{align*}
 \textstyle \sum_i (g_i)_{(1)} \otimes (g_i)_{(2)} \otimes p_i
 &\; =\; e_{H \otimes H, A}(1_H \otimes 1_H \otimes 1_A)
 \\
  &\; \overset{\textnormal{\eqref{eq:brmod2}}}{=}\; (\ide_H \otimes e_{H, A})  (c_{H,H} \otimes \ide_{A}) 
(\ide_{H} \otimes e_{H, A})(c_{H,H}^{-1} 
\otimes \ide_{A})(1_H \otimes 1_H \otimes 1_A) \\
&\;\overset{\textnormal{Lem.\;\ref{lem:R-c}}}{=}\; \textstyle \sum_m (\ide_{H} \otimes e_{H, A})  (c_{H,H} \otimes \ide_{A}) 
(\ide_{H} \otimes e_{H, A})(s^m \otimes t^m \otimes 1_A) \\[.2pc]
&\; =\; \textstyle \sum_{l,m}
(\ide_{H} \otimes e_{H, A})  (c_{H,H} \otimes \ide_{A})(s^m \otimes g_l t^m \otimes p_l)\\
&\; \overset{\textnormal{Lem.\;\ref{lem:R-c}}}{=} \; \textstyle \sum_{k,l,m}
(\ide_{H} \otimes e_{H, A}) (t_k g_l t^m \otimes s_k s^m \otimes p_l)
\\
&\; =\; \textstyle \sum_{j,k,l,m}\;
t_k g_l t^m \otimes g_j s_k s^m \otimes p_j p_l.
\end{align*}
}

Likewise to establish~\eqref{eq:K3}, we compute
{\small
\begin{align*}
 \textstyle \sum_i g_i  \otimes (p_i)_{[-1]} \otimes (p_i)_{[0]}
 &\; =\; e_{H,  H \otimes A}(1_H \otimes 1_H \otimes 1_A)
 \\
  &\; \overset{\textnormal{\eqref{eq:brmod1}}}{=}\; (c_{H,H} \otimes \ide_{A})( \ide_{H} \otimes e_{H,A}) 
(c_{H,H} \otimes \ide_{A}) (1_H \otimes 1_H \otimes 1_A)\\
&\; \overset{\textnormal{Lem.\;\ref{lem:R-c}}}{=} \; \textstyle \sum_{j,k,l}\;
 t_j g_k s_l  \otimes  s_j t_l  \otimes p_k.
\end{align*}
}

Finally, to show that \eqref{eq:K1} holds, recall that $e_{X,M} : X \otimes M \to X \otimes M$ is an $A$-module homomorphism. Therefore,
{\small
\begin{align*}
\textstyle \sum_i g_i a_{[-1]} \otimes p_i a_{[0]} 
& \; =  \;e_{H, A}(a_{[-1]} \otimes a_{[0]}) 
 \;\; =  \;e_{H, A}(a \; \widetilde{\ast} \; (1_H \otimes 1_A))\\[.2pc]
& \;=  \;a \; \widetilde{\ast} \; e_{H, A}(1_H \otimes 1_A)
 \;=  \;a \; \widetilde{\ast} \; (\textstyle \sum_i g_i \otimes p_i) 
 \;=  \; \textstyle \sum_i a_{[-1]} g_i \otimes a_{[0]} p_i.
\end{align*}
}

For part (c), it remains to show that any braiding $e$ on $\lmod{A}$ is given by the action of some element $K\in H\otimes A$. This follows from a reconstruction argument as in \cite{Majid}*{Section 9.4}.
\end{proof}

%\cw{[old version of proof of part (c) now in comments]}
%\rl{To prove Part (c), it remains to show that any braiding $e$ on $\lmod{A}$ is given by the action of some element $K\in H\otimes A$. To show this, we use a reconstruction argument as in \cite{Majid}*{Section 9.4}. We denote by 
%$$F_H\colon \lmod{H}\to \Vect_\Bbbk, \quad F_A\colon \lmod{A}\to \Vect_\Bbbk$$ 
%the forgetful functors. Now, the element $K$ is obtained from an isomorphism 
%between the $\Bbbk$-vector space of natural endomorphisms of  $F_H\otimes F_A$ and elements of $H\otimes A$ which we sketch here for the readers' convenience (see also \cite{Majid}*{Example 9.1.1}). 
%Given a natural transformation 
%$$e\colon F_H\otimes F_A\to F_H\otimes F_A,$$
%define 
%$$K:=e_{H\otimes A}(1_H\otimes 1_A)=\sum_{j}g_j\otimes p_j.$$
%Conversely, given such an element $K$ one defines a natural endomorphism $e$ as in Part (b). 
%Then, for any $H$-module $V$ and any $A$-module $W$, $v\in V, w\in W$, we consider the morphisms 
%$$f_v\colon H\to V, h\mapsto h\cdot v, \qquad  f_w\colon A\to W, a\mapsto a\ast w.$$
%Naturality of $e$ implies that 
%$$e_{V\otimes W}(v\otimes w)=e_{V\otimes W}(l_v\otimes l_w)(1_H\otimes 1_A)=(l_v\otimes l_w)e_{H\otimes A}(1_H\otimes 1_A)=\sum_j g_j\cdot v\otimes p_j\ast w.$$
%Further, evaluating $e$ defined in (b) at $1_H\otimes 1_A$ clearly recovers $K$.
%Now, apply this isomorphism to the braiding $e$ viewed as a natural endomorphism %of $F_H\otimes F_A$ by forgetting the $H$ and $A$-actions.
%}

\begin{remark}
\label{rem:K-matrix-equiv} Let us compare the definition of the quantum $K$-matrix above with that in \cite{Kolb20}. Assume that $H$ is a ribbon Hopf algebra; that is, it is quasitriangular with quantum $R$-matrix $R$ and contains an invertible central element $v$ such that 
$\Delta(v) = (v \otimes v) \big( R_{21}R_{12} \big)^{-1}$ and $v = S(v)$. 
Set
\[
\widetilde{K} := K (v^{-1} \otimes 1_A) \in H \otimes A. 
\]
\begin{enumerate}[\upshape (a)]
\item Analogous to Remark \ref{rem:equiv-e}(a), one shows that conditions \eqref{eq:K2}--\eqref{eq:K1} are equivalent to 
\begin{align}
(\Delta \otimes \ide_A) \widetilde{K} &= \widetilde{K}_{23}R_{21} \widetilde{K}_{13} R_{21},
\label{eq:K2a}
\\
(\ide_H \otimes \delta) \widetilde{K} &= R_{21} \widetilde{K}_{13} R_{21},\label{eq:K3a}
\\
 \widetilde{K} \delta(a) &= \delta(a) \widetilde{K}, \quad \forall a \in A.
\label{eq:K1a}
\end{align}
In turn, \eqref{eq:K2a}--\eqref{eq:K1a} are equivalent to the same set of conditions with \eqref{eq:K2a} replaced by
\begin{equation}
\label{eq:K2aa}
(\Delta \otimes \ide_A) \widetilde{K} = R_{21} \widetilde{K}_{13} R_{21} \widetilde{K}_{23}.
\end{equation}
Indeed, the forward direction follows from
\[
\begin{array}{rll}
(\Delta \otimes \ide_A) \widetilde{K} &\overset{\textnormal{\eqref{eq:K2a}}}{=} \widetilde{K}_{23}R_{21} \widetilde{K}_{13} R_{21} &\overset{\textnormal{\eqref{eq:K3a}}}{=} 
\widetilde{K}_{23} \big( (\ide_H \otimes \delta) \widetilde{K} \big) 
\\[.2pc]
&\overset{\textnormal{\eqref{eq:K1a}}}{=}
\big( (\ide_H \otimes \delta) \widetilde{K} \big) \widetilde{K}_{23}
& \overset{\textnormal{\eqref{eq:K3a}}}{=}  R_{21} \widetilde{K}_{13} R_{21} \widetilde{K}_{23},
\end{array}
\]
while the opposite direction is obtained by reversing the argument. 

\smallskip

\item Conditions \eqref{eq:K3a}, \eqref{eq:K1a},  and \eqref{eq:K2aa} are precisely the conditions for a quasitriangular comodule algebra used in \cite{Kolb20}*{Definition 2.7}, with the only difference that \cite{Kolb20} works with right comodule algebras, while we work with left ones. 

\smallskip

\item Equating the right hand sides of \eqref{eq:K2a} and \eqref{eq:K2aa} gives that 
\[
\widetilde{K}_{23}R_{21} \widetilde{K}_{13} R_{21} = 
R_{21} \widetilde{K}_{13} R_{21} \widetilde{K}_{23},
\]
from which it follows that $\widetilde{K}$ and $R$ define representations of the braid groups of type $B$.
\end{enumerate}
\end{remark}

%%%%%%%%%%%%%%%%%%%%%%%%%%
%%%%%%%%%%%%%%%%%%%%%%%%%%
%%%%%%%%%%%%%%%%%%%%%%%%%%

\section{Reflective centers of module categories} \label{sec:refl-cent}
In this part, we introduce and study the reflective center of a module category $\cM$ over a braided monoidal category $\cC$. Preliminary results on this construction are presented in Section~\ref{sec:ECM}. Then, in Section~\ref{sec:refcent-bimod}, we realize the reflective center of $\cM$ as a center of a certain $\cC$-bimodule category. This enables us to establish properties of reflective centers such as being abelian, finite, and semisimple in the case when $\cC$ is a braided tensor category.

\subsection{Preliminaries on reflective centers} \label{sec:ECM}

We introduce the terminology below. 

\begin{definition} \label{def:ECM}
Let $(\cC, \otimes, \one, a, l, r, c)$ be a braided monoidal category, and let $(\cM, \act, m, \lambda)$ be a left $\cC$-module category.
The {\it reflective center of $\cM$ with respect to $\cC$}   is a category $\cE_{\cC}(\cM)$ defined as follows. 
\begin{enumerate}[(a)]
    \item Its objects are pairs $(M, e^M)$, where $M$ is an object of $\cM$, and 
    \[
    e^M:=\{e^M_Y\colon  Y \act M \overset{\sim}{\to} Y \triangleright M\}_{Y \in \cC} \quad \text{({\it reflection})}
    \]
    is a natural isomorphism such that  $e^M_{X \otimes Y} \; (= e_{X \otimes Y, M})$ satisfies \eqref{eq:brmod2} for all $X,Y \in \cC$. 
    
\smallskip

\item The morphisms $(M, e^M) \to (N, e^N)$ are given by morphisms $f \in \Hom_{\cM}(M,N)$ such that, for all $Y \in \cC$: $$(\ide_Y \act f) \circ e^M_Y = e^N_Y \circ (\ide_Y \act f).$$ 
\end{enumerate}
\end{definition}

\begin{lemma} \label{lem:ECM-modcat}
 Retain the notation above. We have that $\cE_\cC(\cM)$ is a left $\cC$-module category, where by the abusing notation $\act$, the action bifunctor $\act: \cC \times \cE_{\cC}(\cM) \to \cE_{\cC}(\cM)$ is defined by 
$$Y  \act  (M,e^M):=(Y \act M,\; e^{Y \act M}),$$  
where $e^{Y \act M}_X \; (= e_{X, Y \act M})$ is defined by \eqref{eq:brmod1} 
for all $X,Y \in \cC$, and the associativity isomorphism is that of the left $\cC$-module category $\cM$.
\end{lemma}

\begin{proof}
First, we need to show that $\act$ is well-defined on objects and on morphisms. 

Given objects $W,X,Y \in \cC$ and $M \in \cM$, we need to show that $e^{Y \act M}_{W \otimes X}$ defined by~\eqref{eq:brmod1} satisfies~\eqref{eq:brmod2} as in Definition~\ref{def:ECM}(a). This is achieved by the following computation.
{\small
\begin{align*}
e^{Y \act M}_{W \otimes X} &= 
m_{W \otimes X,Y,M} \; (c_{Y,W \otimes X} \act \ide_M) \; m_{Y,W \otimes X,M}^{-1} \; (\ide_Y \act e_{W \otimes X}^M) \; m_{Y,W \otimes X,M} \; (c_{W \otimes X,Y} \act \ide_M) \; m_{W \otimes X,Y,M}^{-1}\\[4pt]
 &= 
m_{W \otimes X,Y,M} \; (c_{Y,W \otimes X} \act \ide_M) \; m_{Y,W \otimes X,M}^{-1} \;  [\ide_Y \act  \left(m_{W,X,M}^{-1} \; (\ide_W \act e_X^M) \; m_{W,X,M} \; (c_{X,W} \act \ide_M)\right)] \smallskip \\
  &\quad  \circ [\ide_Y \act  \left(m_{X,W,M}^{-1} \; (\ide_X \act e_W^M) \; m_{X,W,M}\; (c_{X,W}^{-1} \act \ide_M)\right)] \; m_{Y,W \otimes X,M} \; (c_{W \otimes X,Y} \act \ide_M) \; m_{W \otimes X,Y,M}^{-1}\\[4pt]
  &=  m_{W,X,{Y \act M}}^{-1} \; [\ide_W \act \left( m_{X,Y,M} \; (c_{Y,X} \act \ide_M) \; m_{Y,X,M}^{-1} \; (\ide_Y \act e_X^M)\right)] \\
  & \quad  \circ [\ide_W \act \left(m_{Y,X,M} \; (c_{X,Y} \act \ide_M) \; m_{X,Y,M}^{-1} \right)] \; m_{W,X,{Y \act M}} \; (c_{X,W} \act \ide_{Y \act M})
 \smallskip \\
 &\quad \circ m_{X,W,{Y \act M}}^{-1} \; [\ide_X \act \left( m_{W,Y,M} \; (c_{Y,W} \act \ide_M) \; m_{Y,W,M}^{-1} \; (\ide_Y \act e_W^M)\right)] \\
  & \quad \circ [\ide_X \act \left(m_{Y,W,M} \; (c_{W,Y} \act \ide_M) \; m_{W,Y,M}^{-1} \right)] \; m_{X,W,{Y \act M}}\; (c_{X,W}^{-1} \act \ide_{Y \act M}) \\[8pt]
 &=  m_{W,X,{Y \act M}}^{-1} \; (\ide_W \act e_X^{Y \act M}) \; m_{W,X,{Y \act M}} \; (c_{X,W} \act \ide_{Y \act M}) \; m_{X,W,{Y \act M}}^{-1} \; (\ide_X \act e_W^{Y \act M})  \smallskip \\
 &\quad \circ m_{X,W,{Y \act M}}\; (c_{X,W}^{-1} \act \ide_{Y \act M}).     
\end{align*}
}

\noindent Here, the first and last equations hold by \eqref{eq:brmod1}; the second equation by Definition~\ref{def:ECM}(b) for $e^M$; and the third equation follows from  \eqref{eq:modassoc}, from the braid axiom \eqref{eq:braid1}, and the naturality of $c$. A similar computation shows that the original associativity isomorphism $m_{X,Y,M}$ indeed defines a morphism in $\cE_\cC(\cM)$.

Next, given an object $X \in \cC$, along with  morphisms $f:Y \to Y'$ in $\cC$ and $g:M \to M'$ in $\cM$, we need to show that $(\ide_X \act (f \act g))\circ e^{Y \act M}_X = e^{Y' \act M'}_X \circ (\ide_X \act (f \act g))$ as in Definition~\ref{def:ECM}(b). This is done as follows.
{\small
\begin{align*}
&(\ide_X \act (f \act g))\; e^{Y \act M}_X\\[2pt]
&=
(\ide_X \act (f \act g))\; m_{X,Y,M} \; (c_{Y,X} \act \ide_M) \; m_{Y,X,M}^{-1} \; (\ide_Y \act e_X^M)\; m_{Y,X,M} \; (c_{X,Y} \act \ide_M) \; m_{X,Y,M}^{-1}\\[2pt]
&=
m_{X,Y',M'} \; (c_{Y',X} \act \ide_M) \; m_{Y',X,M'}^{-1} \; (f \act (\ide_X \act g))\circ (\ide_Y \act e_X^M)\; m_{Y,X,M} \; (c_{X,Y} \act \ide_M) \; m_{X,Y,M}^{-1}\\[2pt]
&=
m_{X,Y',M'} \; (c_{Y',X} \act \ide_M) \; m_{Y',X,M'}^{-1} \; (\ide_{Y'} \act e_X^{M'})\;  (f \act (\ide_X \act g)) \; m_{Y,X,M} \; (c_{X,Y} \act \ide_M) \; m_{X,Y,M}^{-1}\\[2pt]
&= m_{X,Y',M'} \; (c_{Y',X} \act \ide_{M'}) \; m_{Y',X,M'}^{-1} \; (\ide_{Y'} \act e_X^{M'})\; m_{Y',X,M'} \; (c_{X,Y'} \act \ide_{M'}) \; m_{X,Y',M'}^{-1} \; (\ide_X \act (f \act g))\\[2pt]
&=e^{Y' \act M'}_X \; (\ide_X \act (f \act g)).
\end{align*}
}

\noindent Here, the first and last equations hold by \eqref{eq:brmod1}; the second and fourth equations hold by the naturality of  $m$ and of  $c$; and the third equation holds by Definition~\ref{def:ECM}(b) for $e^M$.

Therefore, the $\cC$-action bifunctor $\act$ for $\cE_\cC(\cM)$ is well-defined. It also satisfies  \eqref{eq:modassoc} and \eqref{eq:modunit} because they are satisfied for the $\cC$-action bifunctor for $\cM$. 
\end{proof}

An important feature of $\cE_{\cC}(\cM)$ is that the reflections $e^M$ in Definition~\ref{def:ECM}(a) equip this module category with a braiding (as in Section~\ref{sec:bmc}). 

\begin{proposition} \label{prop:ECM-brmodcat}
Take  a braided monoidal category $(\cC,c)$ and a left $\cC$-module category $\cM$. Then,  the reflective center $\cE_{\cC}(\cM)$ is a braided left $\cC$-module category, where
\[
\xymatrix@C=4pc{
e_{Y,(M, e^M)}:\;  Y \act (M,e^M)
\ar[r]^(.58){e^M_Y}
& Y \act (M,e^M),
}
\]
for $Y \in \cC$ and $(M,e^M) \in \cE_\cC(\cM)$. Here, $Y \act (M,e^M):= (Y \act M, \;e^{Y \act M})$ by Lemma~\ref{lem:ECM-modcat}.
\end{proposition}

\begin{proof}
We have that $\cE_{\cC}(\cM)$ is a left $\cC$-module category by Lemma~\ref{lem:ECM-modcat}. So, it suffices to show that $e_{Y,(M, e^M)}:=e^M_Y$ is a braiding for $\cE_{\cC}(\cM)$. 
First,  we verify that $e^M_Y$ is a morphism in $\cE_{\cC}(\cM)$. We compute that, for all $X,Y\in \cC$ and $M \in \cM$: 
{\small
\begin{align*}
&(\ide_X \act e^M_Y) \; e^{Y \act M}_X\\[2pt] 
&= (\ide_X \act e_Y^M) \; m_{X,Y,M} \; (c_{Y,X} \act \ide_M) \; m_{Y,X,M}^{-1} \; (\ide_Y \act e_X^M) \; m_{Y,X,M} \; (c_{X,Y} \act \ide_M) \; m_{X,Y,M}^{-1}\\[2pt]
&= m_{X,Y,M}  \; e_{X \otimes Y}^M \; (c_{Y,X} \act \ide_M)  \; (c_{X,Y} \act \ide_M) \; m_{X,Y,M}^{-1}\\[2pt]
&= m_{X,Y,M}  \; (c_{Y,X} \act \ide_M) \;  e_{Y \otimes X}^M \; (c_{X,Y} \act \ide_M) \; m_{X,Y,M}^{-1}\\[2pt]
&= m_{X,Y,M} \; (c_{Y,X} \act \ide_M) \; m^{-1}_{Y,X,M} \; (\ide_Y \act e_X^M)\; m_{Y,X,M} \;(c_{X,Y} \act \ide_M) \; m^{-1}_{X,Y,M} \; (\ide_X \act e_Y^M)\\[2pt]
&= e^{Y \act M}_X \; (\ide_X \act e^M_Y).
\end{align*}
}

\noindent The first and last equations hold  by Lemma~\ref{lem:ECM-modcat}. The second and fourth equations hold by Definition~\ref{def:ECM}(a). The third equation holds by the naturality of $c$ in the first slot. 
 
 Now we are done since the morphism $e^M_Y$ of $\cE_{\cC}(\cM)$ is an isomorphism by Definition~\ref{def:ECM}(a), and it also satisfies the first braided module category axiom \eqref{eq:brmod2} by Definition~\ref{def:ECM}(a) and the second braided module category axiom \eqref{eq:brmod1} by Lemma~\ref{lem:ECM-modcat}.
\end{proof}

\subsection{Reflective centers as  centers of bimodule categories} \label{sec:refcent-bimod}

Given a braided monoidal category $(\cC,c)$, any left $\cC$-module category $(\cM,\triangleright,m,\lambda)$ is also a right $\cC$-module category $(\cM,\triangleleft,n,\rho)$, where for all $X,Y\in \cC$, $M\in \cM$, we define
\begin{gather*}
    M\triangleleft X:= X\triangleright M,
\end{gather*}
and we define the structure morphisms $n_{M,X,Y}$, $\rho_{M}$, $b_{X,M,Y}$ as follows:
\[
\xymatrix@C=4pc{
n_{M,X,Y} : \; (X \otimes Y) \act M
\ar[r]^(.57){c_{X,Y} \triangleright \ide_{M}}
& 
(Y \otimes X) \act M
\ar[r]^(.49){m_{Y,X,M}}
& 
Y \act (X \act M),
}
\]

\vspace{-.15in} 

\[
\xymatrix@C=4pc{
\rho_M:\; \one \act M
\ar[r]^(.59){\lambda_M}
& 
M,
}
\]

\vspace{-.2in} 

\[
\xymatrix@C=3.8pc{
b_{X,M,Y} : \; 
Y \act (X \act M)
\ar[r]^(.57){m_{Y,X,M}^{-1}}
& 
(Y \otimes X) \act M
\ar[r]^(.49){c_{X,Y}^{-1} \triangleright \ide_{M}}
& 
(X \otimes Y) \act M
\ar[r]^(.48){m_{X,Y,M}}
&
X \act (Y \act M).
}
\]

\vspace{.05in}

\noindent We denote the data $\left(\cM, \; \act, \; \triangleleft:=\act, \;m, \;n:=m(c \triangleright \ide), \;\lambda, \;\rho:=\lambda, \; b:=m(c^{-1} \triangleright \ide)m^{-1}\right)$ by $\cM_{\textnormal{bim}}$, and this is referred to as a {\it one-sided bimodule category}.

\begin{lemma}\label{lem:Mbimodule}
Given the setting above, we have the following statements.
\begin{enumerate}[\upshape (a)]
\item $\cM_{\textnormal{bim}}$ is a $\cC$-bimodule category, and thus is a left $(\cC \boxtimes \cC^{\otimes \textnormal{op}})$-module category. 
\smallskip
\item If $\cM$ is an exact left $\cC$-module category, then $\cM_{\textnormal{bim}}$ is an exact left $(\cC \boxtimes \cC^{\otimes \textnormal{op}})$-module category. 
\end{enumerate}
\end{lemma}

\begin{proof}
Part~(a) follows from \cite{Greenough}*{Proposition~7.1}; see also Remark~\ref{rem:bimod}.
%\cw{[See ``Bimod-cat-over-braided-cat-lemma-20230405"]}
Part~(b) follows from remarks in \cite{DavydovNikshych2013}*{Equation~18}. 
\end{proof}

Next, recall the notion of a center of a bimodule category from Definition~\ref{def:ZCM}, and consider the connection to reflective centers below.

\begin{proposition} \label{prop:ECM-ZCM}
Retain the notation above.  Then we have that  $\cE_{\cC}(\cM)$ and $\cZ_{\cC}(\cM_{\textnormal{bim}})$ are isomorphic as categories.
\end{proposition}

\begin{proof}
Given an object $(M,e^M) \in \cE_{\cC}(\cM)$, we also get that $(M,d^M:=e^M) \in \cZ_{\cC}(\cM_{\textnormal{bim}})$ by setting   $\triangleleft:=\act$, $n:=m(c \triangleright \ide)$, $\rho:=\lambda$, $b:=m(c^{-1} \triangleright \ide)m^{-1}$ (via Lemma~\ref{lem:Mbimodule}). Indeed, \eqref{eq:dXYM} holds by the naturality of $e^M$ and by \eqref{eq:brmod2} as follows:
{\small
\begin{align*}
d_{X \otimes Y}^M &:= e_{X \otimes Y}^M \; = (c^{-1}_{X,Y} \act \ide_M) \; e_{Y \otimes X}^M \; (c_{X,Y} \act \ide_M)\\[.3pc]
&= (c^{-1}_{X,Y} \act \ide_M) \; m_{Y,X,M}^{-1} \; (\ide_Y \act e_X^M) \; m_{Y,X,M} \; (c_{X,Y} \act \ide_M)
\; m_{X,Y,M}^{-1}\\
& \quad  \circ (\ide_X \act e_Y^M) \; m_{X,Y,M}\; (c_{X,Y}^{-1} \act \ide_M)   \; (c_{X,Y} \act \ide_M)\\[.3pc]
& = n_{M,X,Y}^{-1} \; (d_{X}^{M}\triangleleft\ide_Y) \; b^{-1}_{X,M,Y} \; (\ide_X\act d_{Y}^{M}) \; m_{X,Y,M}.
\end{align*}
}

Conversely, given $(M,d^M) \in \cZ_{\cC}(\cM_{\textnormal{bim}})$, we obtain that $(M,e^M:=d^M)$ is in $\cE_{\cC}(\cM)$ by a similar argument. 
This identification of objects extends to an identification of morphisms in  $\cZ_{\cC}(\cM_{\textnormal{bim}})$ (see Definition~\ref{def:ZCM}) with morphisms in $\cE_{\cC}(\cM)$  (see Definition~\ref{def:ECM}(b)). Thus, we have an isomorphism of categories: $\cE_{\cC}(\cM) \cong \cZ_{\cC}(\cM_{\textnormal{bim}})$.
%\cw{[From ``Bimodule centers as braided module categories.pdf"]}
\end{proof}

\begin{corollary} \label{cor:ECMprops} We have the following statements about the reflective center $\cE_{\cC}(\cM)$,  for $\cC$ a braided  tensor category and $\cM$ a left $\cC$-module category.
\begin{enumerate}[\upshape(a)]
\item $\cE_{\cC}(\cM)$ is a $\cZ(\cC)$-module category. 

\smallskip
\item $\cE_{\cC}(\cM) \simeq \mathsf{Fun}_{\cC \boxtimes \cC^{\otimes \textnormal{op}}}(\cC, \cM_{\textnormal{bim}})$ as $\cZ(\cC)$-module categories.

\smallskip

\item $\cE_{\cC}(\cM)$ is abelian when $\cM$ is exact and finite.

\smallskip

\item $\cE_{\cC}(\cM)$ is finite when $\cC$ is finite and $\cM$ is exact and finite.

\smallskip

\item $\cE_{\cC}(\cM)$ is semisimple when $\cC$ and $\cM$ are finite and semisimple.

\end{enumerate}
\end{corollary}

\begin{proof}
It follows from Proposition~\ref{prop:ECM-ZCM} that it suffices to establish the statements for  $\cZ_{\cC}(\cM_{\textnormal{bim}})$. Part~(a) then holds by \cite{Greenough}*{Lemma~7.8}. We can then apply Proposition~\ref{prop:ECM-ZCM} to obtain the action of $\cZ(\cC)$ on $\cE_\cC(\cM)$ below:
{\small
\begin{align*}
\widetilde{\act}: \cZ(\cC) \;\times \;\cE_{\cC}(\cM) &\longrightarrow \cE_{\cC}(\cM) \\
\left( (V, c^V), \; (M, e^M) \right) & \mapsto 
\left( V \otimes M, e^{V \act M}\right),
\end{align*}
}
where,  for any $X \in \cC$, we have
{\small
\begin{align*}
 e^{V \act M}_X &:= 
b_{V,M,X}^{-1} \; (\ide_V \act e_X^M) \; m_{V,X,M} \;(c_X^V \act \ide_M) \; m_{X,V,M}^{-1}\\[.2pc]
&= 
m_{X,V,M} \; (c_{V,X} \act \ide_M)\; m^{-1}_{V,X,M} \;(\ide_V \act e_X^M) \; m_{V,X,M} \; (c_X^V \act \ide_M) \;m_{X,V,M}^{-1}.
\end{align*}
}

For part (b), note that by  \cite{Greenough}*{Proposition~7.10}, $\cZ_{\cC}(\cM_{\textnormal{bim}})$ is isomorphic to $\mathsf{Rex}_{\cC \boxtimes \cC^{\otimes^\textnormal{op}}}(\cC, \cM_{\textnormal{bim}})$. We have that $\cC$ is an exact module category over $\cC \boxtimes \cC^{\otimes^\textnormal{op}}$ by Example~\ref{ex:Cexact}. Now by Proposition~\ref{prop:FunCRexC}, we have that $\mathsf{Rex}_{\cC \boxtimes \cC^{\otimes^\textnormal{op}}}(\cC, \cM_{\textnormal{bim}}) =  \mathsf{Fun}_{\cC \boxtimes \cC^{\otimes^\textnormal{op}}}(\cC, \cM_{\textnormal{bim}})$. So, the result holds.

By part~(b), it suffices to establish parts (c,d,e) for   $\mathsf{Fun}_{\cC \boxtimes \cC^{\otimes^\textnormal{op}}}(\cC, \cM_{\textnormal{bim}})$. Parts~(c,d) then follow from Example~\ref{ex:Cexact}, Lemma~\ref{lem:Mbimodule}(b), Propositions~\ref{prop:FunCabelfin},~\ref{prop:FunCRexC}. Part~(e) follows from Proposition~\ref{prop:FunCss}.
\end{proof}

We note that the $\cZ(\cC)$-module structure  $\widetilde{\act}$ on $\cE_\cC(\cM)$ in \Cref{cor:ECMprops}(a) does not coincide with the one obtained by restricting the (braided) $\cC$-module structure $\act$ of \Cref{lem:ECM-modcat} and \Cref{prop:ECM-brmodcat} along the tensor functor $\cZ(\cC)\to \cC$. However, the $\cC$-action $\act$ can be recovered from $\widetilde{\act}$ by restriction along the tensor functor
$\cC\to \cZ(\cC), \, V\mapsto (V,c^V),$
where $c^V_{X} := c_{X,V}$, for all $X\in\cC$.

%%%%%%%%%%%%%%%%%%%%%%%%%%
%%%%%%%%%%%%%%%%%%%%%%%%%%
%%%%%%%%%%%%%%%%%%%%%%%%%%

\section{Reflective algebras of comodule algebras} \label{sec:rep-refl-cent}
 
In this section, we consider the case when $\cC$ is the braided monoidal  category $\lmod{H}$ for $H$ a quasitriangular Hopf algebra over $\Bbbk$, and $\cM$ is the left $\cC$-module category $\lmod{A}$, for a left $H$-comodule $A$ over $\Bbbk$. (Note that every indecomposable, exact $\cC$-module category is of this form in the finite tensor case \cite{AndruskiewitschMombelli}*{Proposition~1.19}, i.e., when restricting to finite-dimensional modules over finite-dimensional $A$ and $H$.) The module associativity isomorphism $m$ given by the (trivial) associativity isomorphism of $\Vect$. 
The goal of this section is to describe an $H$-comodule algebra $R_H(A)$ that represents the reflective center $\cE_\cC(\cM)$, that is, to get:
\[
\cE_{\lmod{H}}(\lmod{A}) \; \cong \;  \lmod{R_H(A)} \quad \text{as categories}.
\]
The notation for the section and for the rest of the paper is summarized in Section~\ref{sec:setting-repr}. An intermediate category of Doi--Hopf modules is introduced in Section~\ref{sec:equivalence} towards achieving the isomorphism above. With this, we define $R_H(A)$ and establish the desired isomorphism in Section~\ref{sec:equivalence2}. Properties of $R_H(A)$ are examined in Section~\ref{sec:prop-ref}.

\subsection{Standing notation and hypotheses for the Hopf setting}
\label{sec:setting-repr}
We collect for the reader notation and setting that we will use from now on. We use (sumless) Sweedler notation throughout. 
\begin{itemize}
\item $\otimes$ will denote the tensor product ($\otimes_{\Bbbk}$),  monoidal product ($\otimes$), and $\cC$-action bifunctor ($\act$) above from now on as they are all equal $\otimes_{\Bbbk}$.
\medskip
\item $H:= (H,m,u,\Delta, \varepsilon,S)$ is a quasitriangular  Hopf algebra over $\Bbbk$.
\medskip
\item $\Delta(h) = h_{(1)} \otimes h_{(2)}$ is the coproduct of $H$. Its composition is denoted by 
\[(\Delta \otimes \ide_H) \Delta(h) = (\ide_H \otimes \Delta) \Delta (h) =: h_{(1)} \otimes h_{(2)} \otimes h_{(3)}.\]
\item If $H$ is finite-dimensional, then $\langle \hspace{0.02in}, \rangle$ is the Hopf pairing between $H^*$ and $H$. That is,
\begin{equation} \label{eq:Hopfpair}
\langle \xi \zeta, h  \rangle = \langle \xi \otimes \zeta, \Delta(h) \rangle = \langle \xi, h_{(1)} \rangle  \langle \zeta, h_{(2)} \rangle, \quad \langle \xi, h \ell  \rangle = \langle \Delta(\xi), h \otimes  \ell \rangle = \langle \xi_{(1)}, h \rangle \langle \xi_{(2)},  \ell \rangle,
\end{equation}
for $\xi, \zeta \in H^*$, $h, \ell \in H$. Here, $\xi_{(1)} \otimes \xi_{(2)}$ denotes the coproduct of $\xi \in H^*$.
\medskip
\item If $H$ is finite-dimensional, then we denote the dual basis of $H$ by $\{h_d, \xi_d\}_d$, for $h_d \in H$ and $\xi_d \in H^*$. Namely, we get:
\begin{equation} \label{eq:dual-iden}
h = \textstyle \sum_d \langle \xi_d, h \rangle h_d, \quad \quad \quad \quad \xi = \sum_d \langle \xi, h_d \rangle \xi_d.
\end{equation}
\item If $H$ is finite-dimensional, the standard left and right actions of $H$ on $H^*$ are denoted by $\twoheadrightarrow$ and $\twoheadleftarrow$, respectively. That is, for $h, h' \in H$, $\xi \in H^*$:
\begin{equation} \label{eq:twohead}
\begin{array}{c}
h \twoheadrightarrow \xi=  \langle \xi_{(2)}, h \rangle \xi_{(1)} \quad \text{with } \langle h \twoheadrightarrow \xi, h' \rangle := \langle \xi, h' h \rangle,\\[.3pc]
\xi \twoheadleftarrow h :=  \langle \xi_{(1)}, h \rangle \xi_{(2)}\quad \text{with } \langle \xi \twoheadleftarrow h, h' \rangle := \langle \xi, h h' \rangle.
\end{array}
\end{equation}

\medskip
\item $R:=\sum_i s_i \otimes t_i \; \in H \otimes H$ is the $R$-matrix of $H$.
\medskip
\item $R^{-1}:=\sum_i s^i \otimes t^i \; \in H \otimes H$ is the inverse of the $R$-matrix of $H$.
\medskip
\item $A$ is a left $H$-comodule algebra over $\Bbbk$. 
\medskip
\item $\delta: A \to H \otimes A$, $a \mapsto a_{[-1]} \otimes a_{[0]}$, is the left $H$-coaction of $A$. Its composition is denoted by \[(\Delta \otimes \ide_H) \delta(a) = (\ide_H \otimes \delta) \delta (a) =: a_{[-2]} \otimes a_{[-1]} \otimes a_{[0]}.\] 
\item  $K:=\sum_i g_i \otimes p_i \; \in H \otimes A$ is the $K$-matrix of $A$ when $A$ is quasitriangular.
\medskip
\item  $K^{-1}:=\sum_i g^i \otimes p^i \; \in H \otimes A$ is the inverse of the $K$-matrix of $A$ when $A$ is quasitriangular.
\medskip
\item $\cC$ is the braided monoidal category $\lmod{H}$ over $\Bbbk$, with monoidal product $\otimes$, unit  object~$\Bbbk$, and braiding $c$. The $H$-action for objects of $\cC$ is denoted by a centered dot, $\cdot$.
\medskip
\item $c_{X,Y}:X \otimes Y \overset{\sim}{\to}  Y \otimes X, \; x \otimes y \mapsto \sum_i (t_i \cdot y) \otimes (s_i \cdot x)$, is the braiding of $\cC$ via the $R$-matrix of~$H$,  for $X,Y \in \cC$.
\medskip
\item $c^{-1}_{Y,X}: X \otimes Y \overset{\sim}{\to}  Y \otimes X, \; x \otimes y \mapsto \sum_i (s^i \cdot y) \otimes (t^i \cdot x)$,  is the inverse braiding of $\cC$, via the inverse $R$-matrix of~$H$, for $X,Y \in \cC$.

\medskip
\item $\cM$ is the  left $\cC$-module category $\lmod{A}$ over $\Bbbk$. The $A$-action for objects of $\cM$ is denoted by an asterisk, $\ast$, or by $\widetilde{\ast}$ if the action is induced.
\medskip

\item $e^M$ is the braiding of $\cM = \lmod{A}$ for $M \in \cM$.

\medskip

\item $e^M_X(x \otimes m) := \sum_i (g_i \cdot x) \otimes (p_i \ast m)$, for $X \in \cC$ and $\sum_i g_i \otimes  p_i \in H \otimes A$.

\medskip

\item When $H$ (resp., $A$) is in the subscript of $c$ or $e$,  this indicates  the regular left $H$-module (resp., $A$-module).
\end{itemize}

%%%%%%%%%%%%%%%%%%%%%%%%%%
\subsection{Reflective centers as Doi--Hopf modules}
\label{sec:equivalence}
Here, we will consider a category of left Doi--Hopf modules, ${}^{\widehat{H}}_A \mathsf{DH}(H)$, consisting of vector spaces which are modules over the left $H$-comodule algebra $A$ and comodules over a left $H$-module coalgebra $\widehat{H}$, which is a version of Majid's covariantized (or transmuted) coalgebra \cites{Majid1991,Majid}. Our main goal is to show that 
\begin{equation} \label{eq:ECM-DH}
\cE_\cC(\cM) \; \cong\;  {}^{\widehat{H}}_A  \mathsf{DH}(H)\quad  \text{as categories.}
\end{equation}

We recall the category of Doi--Hopf modules in Section~\ref{sec:Doi--Hopf}; we then get ${}^{\widehat{H}}_A  \mathsf{DH}(H)$ after we define $\widehat{H}$ in Section~\ref{sec:hatH}. Next, we construct a  functor $F: \cE_\cC(\cM) \to   {}^{\widehat{H}}_A   \mathsf{DH}(H)$ in Section~\ref{sec:ECMtoDH},  and a  functor $G: {}^{\widehat{H}}_A   \mathsf{DH}(H) \to  \cE_\cC(\cM) $ in Section~\ref{sec:DHtoECM}. Then, we establish~\eqref{eq:ECM-DH} in Section~\ref{sec:ECM-DH-isocat}.

\subsubsection{The category of Doi--Hopf modules} \label{sec:Doi--Hopf}
First, let us recall the notion of a Doi--Hopf module from work of Doi \cite{Doi1992}. Note that we abuse some of the notation of Section~\ref{sec:setting-repr} below.

\begin{definition} \cite{Doi1992}*{Remark~1.3} 
Consider the following input data:
\begin{itemize}
    \item $L$, a Hopf algebra;
    \smallskip
    \item $B$, a left $L$-comodule algebra with coaction given by $\delta: B \to L \otimes B, \; b \mapsto b_{[-1]} \otimes b_{[0]}$;
    \smallskip
    \item $C$, a left $L$-module coalgebra with action given by $\rightharpoonup: L \otimes C \to C$.
\end{itemize}
A vector space $M$ is a {\it left $(L,B,C)$-Doi--Hopf module} is if the following conditions hold:
\begin{enumerate}[(i)]
    \item $M$ is a left $B$-module with action given by $\ast: B \otimes M \to M$;
    \smallskip
    \item $M$ is a left $C$-comodule with coaction given by $\varphi: M \to C \otimes M, \; m \mapsto m_{-1} \otimes m_{0}$;
\end{enumerate}
action and coaction are subject to the following compatibility condition,
\begin{equation} \label{eq:Doi-Hopf}
(b \ast m)_{-1}  \otimes 
(b \ast m)_{0}
\; = \;
(b_{[-1]} \rightharpoonup m_{-1}) \otimes (b_{[0]} \ast m_0),
\end{equation}
for all $m \in M$ and $b \in B$.
\end{definition}

The collection of left $(L,B,C)$-Doi--Hopf modules forms a category. Here, a morphism between two left $(L,B,C)$-Doi--Hopf modules is a map that is simultaneously a left $B$-module map and a left $C$-comodule map. We denote this category by 
\[{}^C_B  \mathsf{DH}(L).\]

In the appendix of the ArXiv version 1 of this article it is shown that the category of left $(L,B,C)$-Doi--Hopf modules admits a canonical structure of a left module category over the braided monoidal category $\lmod{L}$ when $L$ is quasitriangular, with $R$-matrix $\sum_i s_i \otimes t_i$.

\subsubsection{The left $H$-module coalgebra $\widehat{H}$} 
\label{sec:hatH}
Let us define the $H$-module coalgebra $\widehat{H}$ mentioned above in~\eqref{eq:ECM-DH} which is a version of Majid's transmuted (or covariantized) coalgebra; see \Cref{rem:Majid}.

\begin{definition}\label{def:hatH}
Take $\widehat{H}$ to be equal to $H$ as vector spaces, and consider the following comultiplication, counit, and left $H$-action formulae:
\begin{align*}
\widehat{\Delta}(h) &:= \; 
\textstyle \sum_{i,j} t_j  h_{(1)}  t_i  \;\otimes\;  h_{(2)}  s_i  S^{-1}(s_j), \\
\widehat{\varepsilon}(h) &:= \; \varepsilon(h),\\
\ell \rightharpoonup h &:= \; \ell_{(2)}  h  S^{-1}(\ell_{(1)}),
\end{align*}

\noindent for all $h \in \widehat{H}$ and $\ell \in H$. 
\end{definition}

The operations from \Cref{def:hatH} make $\widehat{H}$ a left $H$-module coalgebra. This follows as in \cite{Majid}*{Theorem~7.4.2}.

\begin{remark} \label{rem:Majid}.
The precise comparison to the conventions of \cite{Majid} is as follows. For a quasitriangular Hopf algebra $H:=(H, R:=\sum_i s_i \otimes t_i)$, we have that its coopposite Hopf algebra $H^{\textnormal{cop}}$ is a quasitriangular Hopf algebra with $R^{\textnormal{cop}}:=\sum_i t_i \otimes s_i$.
With this, and by inspecting the proof of \cite{Majid1991}*{Theorem~3.1}, one can see that $(\widehat{H})^{\textnormal{cop}} = (H^{\textnormal{cop}})_{\textnormal{trm}}$, where $H_{\textnormal{trm}}$ denotes the coalgebra obtained by {\it transmutation} in \cite{Majid}*{Theorem~7.4.2}. 
\end{remark}

%%%%%%%%%%%%%%%%%%%%%%%%%%
\subsubsection{Functor from  the reflective center to a category of Doi--Hopf modules} \label{sec:ECMtoDH}

 Consider the following preliminary result.

\begin{lemma} \label{lem:Lambda}
Let $r_h$ be right multiplication by $h \in H$. Then, the operator  
\begin{align*}
\Lambda: \Hom_\Bbbk(A \otimes M, \widehat{H} \otimes M) &\to \Hom_\Bbbk(A \otimes M, \widehat{H} \otimes M)\\
\psi &\mapsto  [a \otimes m \;\mapsto \; ((r_{a_{[-1]}} \otimes \ide_M) \circ\psi)(a_{[0]} \otimes m)] 
\end{align*}

\vspace{-.05in}

\noindent is invertible, with inverse
\begin{align*}
\Lambda^{-1}: \Hom_\Bbbk(A \otimes M, \widehat{H} \otimes M) &\to \Hom_\Bbbk(A \otimes M, \widehat{H} \otimes M)\\
\psi &\mapsto  [a \otimes m \;\mapsto \; ((r_{S^{-1}(a_{[-1]})} \otimes \ide_M) \circ\psi)(a_{[0]} \otimes m)]. 
\end{align*}

 \vspace{-.27in}

 \qed
\end{lemma}

Next, we turn our attention to the desired functor for this section.

\begin{proposition} \label{prop:functorF}
We have a functor
\begin{align*}
F: \cE_{\lmod{H}}(\lmod{A}) &\to {}_A^{\widehat{H}} \mathsf{DH}(H)\\
(M, \; \ast, \; e^M) &\mapsto (M, \; \ast, \; \varphi:=\varphi_{e^M}: M \to \widehat{H} \otimes M),
\end{align*}

\vspace{-.05in}

\noindent where $\varphi(m) = e_{H}^M(1_H \otimes m)=: m_{-1} \otimes m_0$.
\end{proposition}

\begin{proof}
It suffices to establish the following statements:
\begin{itemize}[(i)]
    \item[(i)] $(M ,\varphi) \in \lcomod{\widehat{H}}$ \; (this will follow from \eqref{eq:brmod2}), and 
    \smallskip 
    \item[(ii)] $(M,\ast, \varphi)$ satisfies \eqref{eq:Doi-Hopf} \; (this will follow from $e_X^M \in \lmod{A}$ for any $X \in \lmod{H}$).  
\end{itemize}

Towards (i), for $h, \ell \in H$ (as the regular left $H$-module), note that
\begin{equation} \label{eq:cHH}
c_{H,H}(h \otimes \ell) = \textstyle \sum_i t_i \ell \otimes s_i h
\quad \quad \text{and} \quad \quad 
c^{-1}_{H,H}(h \otimes \ell) = \textstyle \sum_i s^i \ell \otimes t^i h.
\end{equation}
Moreover by taking $r_h: H \to H$ to be right multiplication by $h$, we get that $r_h \in \lmod{H}$. Thus, by the naturality of $e^M$, we obtain that 
\begin{equation} \label{eq:eHM}
e_{H}^M(h \otimes m) \; = \; 
e_{H}^M(r_h \otimes \ide_M)(1_H \otimes m) \;  = \;
(r_h \otimes \ide_M)e_{H}^M(1_H \otimes m)
\; = \; m_{-1}h \otimes m_0.
\end{equation}
Moreover, note that $\widehat{\Delta} = \omega \circ \Delta$ for 
\begin{equation} \label{eq:omega}
\omega(h \otimes h') := \textstyle \sum_{k,l} t_l  h  t_k \; \otimes \;  h'  s_k  S^{-1}(s_l).
\end{equation}
In fact, as an aside, we have that $\omega$ is invertible with 
\begin{equation} \label{eq:omega-inv}
\omega^{-1}(h \otimes h') := \textstyle \sum_{i,j}  \; t^i  h   t^j \; \otimes \;  h' S^{-1}(s^i)  s^j.
\end{equation}

Now (i) holds by the following computation:
{\small
\begin{align*}
(\widehat{\Delta} \otimes \ide_M)\varphi(m) 
& \; = \; 
(\omega \otimes \ide_M)(\Delta \otimes \ide_M)e_H^M(1_H \otimes m) \\
& \; \overset{e^M \text{nat'l},\  \Delta \in \lmod{H}}{=} \; 
(\omega \otimes \ide_M)e_{H\otimes H}^M(\Delta \otimes \ide_M)(1_H \otimes m)\\[.2pc]
& \; = \; 
(\omega \otimes \ide_M)e_{H\otimes H}^M(1_H \otimes 1_H \otimes m)\\
& \; \overset{\textnormal{\eqref{eq:brmod2}}}{=} \; 
(\omega \otimes \ide_M)(\ide_H \otimes e_H^M)(c_{H,H} \otimes \ide_M)(\ide_H \otimes e_H^M)(c_{H,H}^{-1} \otimes \ide_M)(1_H \otimes 1_H \otimes m)\\
& \; \overset{\textnormal{\eqref{eq:cHH}}}{=} \; 
\textstyle \sum_i(\omega \otimes \ide_M)(\ide_H \otimes e_H^M)(c_{H,H} \otimes \ide_M)(\ide_H \otimes e_H^M)(s^i \otimes t^i \otimes m)\\
& \; \overset{\textnormal{\eqref{eq:eHM}}}{=} \; 
\textstyle \sum_i(\omega \otimes \ide_M)(\ide_H \otimes e_H^M)(c_{H,H} \otimes \ide_M)(s^i \otimes m_{-1}t^i \otimes m_0)\\
& \; \overset{\textnormal{\eqref{eq:cHH}}}{=} \; 
\textstyle \sum_i(\omega \otimes \ide_M)(\ide_H \otimes e_H^M)(t_jm_{-1}t^i \otimes s_js^i  \otimes m_0)
\\
& \; \overset{\textnormal{\eqref{eq:eHM}}}{=} \; 
\textstyle \sum_{i,j}(\omega \otimes \ide_M)(t_jm_{-2}t^i \otimes m_{-1}s_js^i  \otimes m_0)
\\[.2pc]
& \; = \; 
\textstyle \sum_{i,j,k,l} t_l t_j m_{-2} t^i t_k \;  \otimes \; m_{-1} s_j s^i s_k S^{-1}(s_l) \;  \otimes \; m_0\\[.2pc]
& \; = \; 
\textstyle \sum_{j,l} t_l t_j m_{-2}  \;  \otimes \; m_{-1} s_j  S^{-1}(s_l) \;  \otimes \; m_0\\[.2pc]
& \;\overset{\textnormal{\eqref{eq:QT5}}}{=} \; 
  m_{-2}  \;  \otimes \; m_{-1} \;  \otimes \; m_0
\\[.2pc]
& \;  = \; 
(\ide \otimes \varphi)\varphi(m).
\end{align*}
}

 To prove (ii), note that $e_{H}^M \in \lmod{A}$ ($\dag$). So, for $a \in A$ and $m \in M$, we get that:
\begin{align} 
 \begin{split}
\hbox{{\small $a_{[-1]}m_{-1} \; \otimes  \; (a_{[0]} \ast m_0)$}}
&\; \hbox{{\small $= \; 
a \ast (m_{-1} \otimes m_0)
\quad \overset{\textnormal{\eqref{eq:eHM}}}{=} \; 
a \ast e_{H}^M( 1_H \otimes m)$}}\\
&
\hbox{{\small $\; \overset{(\dag)}{=} \; 
e_{H}^M(a \ast ( 1_H \otimes m))
\; = \; 
e_{H}^M(a_{[-1]} \otimes (a_{[0]} \ast m))$}}\\
&
\hbox{{\small $\; \overset{\textnormal{\eqref{eq:eHM}}}{=} \; 
(a_{[0]} \ast m)_{-1} a_{[-1]}\; \otimes\;  (a_{[0]} \ast m)_{0}.$}}
\label{eq:Fii}
\end{split}
\end{align}

\noindent Now the following computation verifies (ii): 
{\small
\begin{align*}
(a_{[-1]} \rightharpoonup m_{-1}) \; \otimes \; (a_{[0]} \ast m_0) 
& \; \overset{\textnormal{Def.\;\ref{def:hatH}}}{=} \; 
(a_{[-1]})_{(2)} \;m_{-1}\; S^{-1}((a_{[-1]})_{(1)})\; \otimes \; (a_{[0]} \ast m_0) \\
& \; = \; 
a_{[-1]} \;m_{-1}\; S^{-1}(a_{[-2]})\; \otimes \; (a_{[0]} \ast m_0) \\
& \; \overset{(\ddag)}{=} \; 
(a_{[0]} \ast m)_{-1} \; a_{[-1]} \; S^{-1}(a_{[-2]})\; \otimes \; (a_{[0]} \ast m)_0 \; \\[.2pc]
& \; = \; 
(a_{[0]} \ast m)_{-1} \; \varepsilon(a_{[-1]})\; \otimes \; (a_{[0]} \ast m)_0 \\[.2pc]
& \; = \; 
(a \ast m)_{-1} \; \otimes \; (a \ast m)_0. 
\end{align*}
}

\noindent At $(\ddag)$, we applied the operator $\Lambda^{-1}$ from Lemma~\ref{lem:Lambda} to \eqref{eq:Fii}.
This concludes the proof. 
\end{proof}

%%%%%%%%%%%%%%%%%%%%%%%%%%
\subsubsection{Functor from a category of Doi--Hopf modules to the reflective center} \label{sec:DHtoECM}

\begin{proposition} \label{prop:functorG}
For $\varphi(m):=m_{-1} \otimes m_0$, we have a functor
\begin{align*}
G: {}_A^{\widehat{H}} \mathsf{DH}(H) &\to \cE_{\lmod{H}}(\lmod{A})\\
\left(M, \; \ast, \; \varphi: M \to \widehat{H} \otimes M\right) &\mapsto \left(M, \; \ast, \; e_X^M:=(e_X^M)_{\varphi}: X \otimes M \to X \otimes M\right)
\end{align*}
where $e_X^M(x \otimes m) = (m_{-1} \cdot x) \otimes m_0$, for $(X, \cdot)$  a left $H$-module.
\end{proposition}

\begin{proof}
It suffices to establish the following statements:
\begin{itemize}
\item[(i)] $e_X^M$   satisfies \eqref{eq:brmod2} \; (this will follow from $(M ,\varphi) \in \lcomod{\widehat{H}}$), and 
\smallskip
\item[(ii)] $e_X^M \in \lmod{A}$ for any $X \in \lmod{H}$ \; (this will follow from \eqref{eq:Doi-Hopf}).
\end{itemize}
\smallskip

To verify (i), recall the invertible linear map $\omega$ from \eqref{eq:omega} and \eqref{eq:omega-inv}, and recall that $\widehat{\Delta} = \omega \circ \Delta$ for the coproduct $\widehat{\Delta}$ in \Cref{def:hatH}. We use the Sweedler notation $\widehat{\Delta}(h):= h_{\widehat{(1)}} \otimes h_{\widehat{(2)}}$ for $h \in \widehat{H}$,  along with $\Delta(h):= h_{(1)} \otimes h_{(2)}$ for $h \in H$. Also, for $X,Y \in \lmod{H}$, and $x \in X$, $y \in Y$, we have that:
\begin{equation} \label{eq:cXY}
c_{Y,X}(y \otimes x) = \textstyle \sum_i (t_i \cdot x) \otimes (s_i \cdot y)
\quad \quad \text{and} \quad \quad 
c^{-1}_{Y,X}(x \otimes y) = \textstyle \sum_i (s^i \cdot y) \otimes (t^i \cdot x).
\end{equation}

Now (i) holds via the computation below for $X,Y \in \lmod{H}$, $M \in \lmod{A}$,  $x \in X$, $y \in Y$, $m \in M$:
\pagebreak
{\small
\begin{align*}
e_{X \otimes Y}^M(x \otimes y \otimes m) 
& \; = \;
(m_{-1} \cdot (x \otimes y)) \; \otimes \; m_0\\[.2pc]
& \; = \;
((m_{-1})_{(1)} \cdot x)\; \otimes \;((m_{-1})_{(2)} \cdot y)\; \otimes\; m_0\\
& \; \overset{\textnormal{\eqref{eq:omega-inv}}}{=} \;
\textstyle \sum_{i,j} \; [t^i\;(m_{-1})_{\widehat{(1)}} \;  t^j \cdot x]\; \otimes \;[(m_{-1})_{\widehat{(2)}}  \; S^{-1}(s^i) \; s^j \cdot y] \; \otimes\; m_0\\
& \; \overset{\varphi\  \in \lcomod{\widehat{H}}}{=} \;
\textstyle \sum_{i,j} \; [t^i\;m_{-2} \;  t^j \cdot x]\; \otimes \;[m_{-1}  \; S^{-1}(s^i) \; s^j \cdot y] \; \otimes\; m_0\\
& \; \overset{\textnormal{\eqref{eq:QT5}}}{=} \;
\textstyle \sum_{i,j} \; [t_i \;m_{-2} \;  t^j \cdot x]\; \otimes \;[m_{-1}  \; s_i \; s^j \cdot y] \; \otimes\; m_0\\
& \; \overset{\textnormal{\eqref{eq:eHM}}}{=} \;
\textstyle \sum_{i,j} (\ide_X \otimes e_Y^M)[ t_i \; m_{-1} \; t^j \cdot x) \otimes (s_i\; s^j \cdot y)\otimes m_0]\\
& \; \overset{\textnormal{\eqref{eq:cXY}}}{=} \;
\textstyle \sum_j (\ide_X \otimes e_Y^M)(c_{Y,X} \otimes \ide_M)[(s^j \cdot y) \otimes (m_{-1} t^j \cdot x) \otimes m_0]\\
& \; \overset{\textnormal{\eqref{eq:eHM}}}{=} \;
\textstyle \sum_j (\ide_X \otimes e_Y^M)(c_{Y,X} \otimes \ide_M)(\ide_Y \otimes e_X^M)[(s^j \cdot y) \otimes (t^j \cdot x) \otimes m]\\
& \; \overset{\textnormal{\eqref{eq:cXY}}}{=} \;
(\ide_X \otimes e_Y^M)(c_{Y,X} \otimes \ide_M)(\ide_Y \otimes e_X^M)(c_{Y,X}^{-1} \otimes \ide_M)(x \otimes y \otimes m).
\end{align*}
}

\smallskip

Towards (ii), note that, for $a \in A$ and $m \in M$,  we get that:
{\small
\begin{align*}
 (a \ast m)_{-1} \otimes (a \ast m)_{0} &\; \overset{\textnormal{\eqref{eq:Doi-Hopf}}}{=} \;
(a_{[-1]} \rightharpoonup m_{-1}) \otimes (a_{[0]} \ast m_0) \nonumber\\
&\; \overset{\textnormal{Def.\;\ref{def:hatH}}}{=} \; 
((a_{[-1]})_{(2)} \; m_{-1} \; S^{-1}((a_{[-1]})_{(1)}) \; \otimes \; (a_{[0]} \ast m_0) \\
%\label{eq:Gii-1}\\
& \; =\;  
a_{[-1]} \; m_{-1} \; S^{-1}(a_{[-2]}) \; \otimes\; (a_{[0]} \ast m_0).\nonumber
\end{align*}
}

\noindent By the counit and antipode axioms, we get that: 
{\small
\begin{align*}
(a_{[0]} \ast m)_{-1} \; a_{[-1]} \; S^{-1}(a_{[-2]}) \otimes (a_{[0]} \ast m)_{0} 
&\; = \;(a_{[0]} \ast m)_{-1} \; \varepsilon(a_{[-1]}) \otimes (a_{[0]} \ast m)_{0}\\
&\; = \;
a_{[-1]} \; m_{-1} \; S^{-1}(a_{[-2]}) \; \otimes\; (a_{[0]} \ast m_0).
\end{align*}
}

\noindent Applying the operator $\Lambda$ from Lemma~\ref{lem:Lambda} to the above equation yields:
{\small
\begin{align*}
(a_{[0]} \ast m)_{-1} \; a_{[-1]} \; S^{-1}(a_{[-2]}) \; a_{[-3]}\otimes (a_{[0]} \ast m)_{0} \; = \;
a_{[-1]} \; m_{-1} \; S^{-1}(a_{[-2]}) \; a_{[-3]} \; \otimes\; (a_{[0]} \ast m_0).
\end{align*}
}

\noindent Using  antipode and counit axioms again yields:
{\small
\begin{align*} 
(a_{[0]} \ast m)_{-1} \; a_{[-1]} \; \otimes \; (a_{[0]} \ast m)_0 \; = \; a_{[-1]} m_{-1} \; \otimes (a_{[0]} \ast m_0).
\end{align*}
}

\noindent Now (ii) follows from the following computation:
{\small
\begin{align*}
e_X^M(a \ast (x \otimes m)) 
& \; = \; 
e_X^M((a_{[-1]} \cdot x) \; \otimes \; (a_{[0]} \ast m)) \hspace{.27in} \overset{\textnormal{\eqref{eq:eHM}}}{=} \; 
((a_{[0]} \ast m)_{-1} \; a_{[-1]} \cdot x) \; \otimes \; (a_{[0]} \ast m)_{0} \\
& \; = \; 
 (a_{[-1]} m_{-1} \cdot x) \; \otimes \; (a_{[0]} \ast m_0) \; = \; 
a \ast ((m_{-1} \cdot x) \;  \otimes \; m_0)\\[.2pc]
& \; = \; 
a \ast e_X^M(x \otimes m).
\end{align*}
}

\noindent This concludes the proof of the result.
\end{proof}

%%%%%%%%%%%%%%%%%%%%%%%%%%
\subsubsection{Isomorphism of categories} \label{sec:ECM-DH-isocat}

Now we establish the category isomorphism \eqref{eq:ECM-DH}. 

\begin{proposition} \label{prop:ECM-DH} We have that the reflective center $\cE_{\lmod{H}}(\lmod{A})$ and the category of Doi--Hopf modules ${}^{\widehat{H}}_A  \mathsf{DH}(H)$  are isomorphic  as  categories.
\end{proposition}

\begin{proof}
It suffices to show that the functors $F: \cE_{\lmod{H}}(\lmod{A}) \to {}^{\widehat{H}}_A  \mathsf{DH}(H)$ from Proposition~\ref{prop:functorF}, and $G: {}^{\widehat{H}}_A  \mathsf{DH}(H) \to \cE_{\lmod{H}}(\lmod{A})$ from Proposition~\ref{prop:functorG}, are mutually inverse. Starting with an object $(M,e^M)$ in $\cE_{\lmod{H}}(\lmod{A})$, consider the half-braiding $\ov{e}$ on $GF(M,e^M)$ given by 
$$\ov{e}^M_X(x\otimes m)=(m_{-1}\cdot x)\otimes m_0,$$
for $X \in \lmod{H}$ and $x \in X$. Here,
$m_{-1}\otimes m_0:=e_H^M(1_H\otimes m)$.
For $x\in X$, take the morphism:
$$l_x\colon H\to X, \quad h\mapsto h\cdot x \quad \in  \lmod{H}.$$
Applying naturality of the braiding $e^M$ to $l_x$ gives
$$\ov{e}^M_X(x\otimes m) \; = \; (l_x\otimes \ide_M)e_H^M(1_H\otimes m)  \;=  \;e_X^M(l_x\otimes \ide_M)(1_H\otimes m)  \;=  \;e_X^M(1_H\cdot x\otimes m)  \;= \; e_X^M(x\otimes m).$$
Hence, $\ov{e}=e$, and the identity $GF(M)\to M$ is a morphism of objects in $\cE_{\lmod{H}}(\lmod{A})$. This shows $GF=\ide_{\cE_{\lmod{H}}(\lmod{A})}$.

On the other hand, for an object $M$ in ${}^{\widehat{H}}_A  \mathsf{DH}(H)$  with coaction $M\to \widehat{H}\otimes M, m\mapsto m_{-1}\otimes m_0$, consider the induced $\widehat{H}$-coaction $\ov{\varphi}$ on $M$ obtained on $FG(M)$. Then
$$\ov{\varphi}(m) \; = \;e_H^{G(M)}(1_H\otimes m) \;= \; (m_{-1}\cdot 1_H)\otimes m_0 \;= \; m_{-1}\otimes m_0.$$
This shows that $(FG(M),\ov{\varphi})=(M,\varphi)$ as Doi--Hopf modules, and hence $FG = \ide_{{}^{\widehat{H}}_A  \mathsf{DH}(H)}$.
\end{proof}

%%%%%%%%%%%%%%%%%%%%%%%%%%

\subsection{Reflective centers represented by reflective algebras} 
\label{sec:equivalence2}
The goal of this section is to extend the isomorphism \eqref{eq:ECM-DH} to the isomorphism below:
\begin{equation} \label{eq:ECM-DH-2}
\cE_\cC(\cM) \; \cong\;  {}^{\widehat{H}}_A  \mathsf{DH}(H) \; \cong\;  \lmod{R_H(A)} \quad  \text{as categories,}
\end{equation}
for some $\Bbbk$-algebra $R_H(A)$ which we call a {\it reflective algebra}. In Section~\ref{sec:double-DH}, we construct a general isomorphism between ${}^C_B  \mathsf{DH}(L)$ (from Section~\ref{sec:Doi--Hopf}) and the category of modules of a crossed product algebra $B \rtimes_L (C^*)^\text{op}$. Next, we study the dual of the $H$-module coalgebra $\widehat{H}$ (from Section~\ref{sec:hatH}) in Section~\ref{sec:hatHstar}, and then define $R_H(A)$ as a crossed product algebra in Section~\ref{sec:refl-alg}.

%%%%%%%%%%%%%%%%%%%%%%%%%%%%%

\subsubsection{Doi--Hopf modules and crossed products} 
\label{sec:double-DH}
In the setting of Section \ref{sec:Doi--Hopf}, consider the case when the left $L$-module coalgebra $C$ is finite-dimensional.
Then $C^*$ is a right $L$-module algebra with the $L$-action given by 
\begin{equation} \label{eq:LactC*}
\langle \xi \leftharpoonup \ell, c \rangle := \langle \xi, \ell \rightharpoonup c \rangle,
\end{equation}
for $\xi \in C^*, c \in C, \ell \in L$.
The same action makes $(C^*)^{\textnormal{op}}$
a right $L^{\textnormal{cop}}$-module algebra.
Define the {\it crossed product algebra} (see \cite{Doi1992}*{Section~1}): 
\[
B \rtimes_L (C^*)^{\textnormal{op}} := B \circledast (C^*)^{\textnormal{op}}/ \left( \xi b - b_{[0]} (\xi \leftharpoonup b_{[-1]}),\;\; b \in B, \; \xi \in (C^*)^{\textnormal{op}} \right), 
\]
where $\circledast$ denotes the free product of algebras. That is, for all $a,b \in B$ and $\xi, \zeta \in (C^*)^{\textnormal{op}}$, we get:
\begin{equation} \label{eq:BC-prod}
\begin{array}{rl}
(a \; \xi)(b \; \zeta) \; &:=\text{mult}^{B \rtimes_L (C^*)^{\textnormal{op}}}\big(a \circledast \xi, \; b \circledast \zeta \big)\\[.4pc]
& \; = \text{mult}^{B}\big(a, b_{[0]}\big) \; \text{mult}^{(C^*)^{\textnormal{op}}}\big(\xi \leftharpoonup b_{[-1]}, \zeta\big)\\[.4pc]
& \; =: a \; b_{[0]}\;  \zeta \; (\xi \leftharpoonup b_{[-1]}).
\end{array}
\end{equation}

\begin{lemma}\label{lem:cross-prod} We have that $B \rtimes_L (C^*)^{\textnormal{op}}$ is isomorphic to $B \otimes (C^*)^{\textnormal{op}}$ as a $(B,(C^*)^{\textnormal{op}})$-bimodule. 
\end{lemma}
\begin{proof} Consider the tensor product $B \otimes (C^*)^{\textnormal{op}}$. It is straight-forward to verify  that the product~\eqref{eq:BC-prod}
on it is associative, where $a, b \in B$ and $\xi, \zeta \in (C^*)^{\textnormal{op}}$. Therefore, $B \rtimes_L (C^*)^{\textnormal{op}}$ is isomorphic to the space $B \otimes (C^*)^{\textnormal{op}}$ with this product. This implies the statement of the lemma.
\end{proof}

We now recall results relating Doi--Hopf modules to modules over the crossed product algebra.

\begin{lemma} \label{lem:C-comod}
 Let $\{c_d,  \xi_d\}_d$ be a dual basis of $C$. Then, we have an isomorphism of categories: 
\[
\Omega: C\text{-}\mathsf{comod} \overset{\sim}{\to} \lmod{(C^*)^{\textnormal{op}}},
\]
given by the assignments, for $\xi \in C^*$ and $m \in M$:
 \[
 {\small
 \hspace{-.1in}
 \begin{array}{c}
 \Big(M, \; \varphi: M \to C \otimes M,\;  \varphi(m):= m_{-1} \otimes m_0\Big) 
 \mapsto 
 \Big(M,\; \star_\varphi: (C^*)^{\textnormal{op}} \otimes M \to  M, \; \xi \star_\varphi m:= \langle \xi, m_{-1} \rangle m_0\Big),\\[.5pc]
 \Big(M, \; \varphi_\star: M \to C \otimes M,\;  \varphi_\star(m):= \sum_d c_d \otimes (\xi_d \star m) \Big) 
 \mapsfrom 
 \Big(M,\; \star: (C^*)^{\textnormal{op}} \otimes M \to  M\Big).
 \end{array}
 }
 \]

 \vspace{-.22in} 
 
 \qed
\end{lemma}

The following result is an analogue of \cite{Doi1992}*{Remark~1.3(b)} with our conventions.

\begin{proposition} \label{prop:isom-DH-double} The functor below is an isomorphism of  categories:
\begin{align*}
\Omega_{L;B,C}&\colon {}^C_B  \mathsf{DH}(L) \to \lmod{(B \rtimes_L (C^*)^{\textnormal{op}})}, \quad
(M, \ast, \varphi) \mapsto (M, \star),
\end{align*}
where $\star$ is defined by $\xi \star m:=\xi \star_\varphi m$ as given in Lemma~\ref{lem:C-comod}, for $\xi\in C^*$, and by $b\star m=b\ast m$, for $b\in B$ and $m\in M$.
\end{proposition}

\begin{proof}
Recall that $M \in {}^C_B  \mathsf{DH}(L)$ is a left $B$-module via the action $\ast : B \otimes M \to M$, and  a left $C$-comodule via the action $\varphi : M \to C \otimes M$, $m \mapsto m_{-1} \otimes m_0$. By dualization, $M$ is a left $(C^*)^{\textnormal{op}}$-module via Lemma~\ref{lem:C-comod}. The two actions on $M$ are compatible in the following way:
{\small
\[
\begin{array}{rll}
\xi \star ( b \ast m) & 
\; = \; \langle \xi, (b \ast m)_{-1} \rangle (b \ast m)_{0} 
& \; \overset{\textnormal{\eqref{eq:Doi-Hopf}}}{=} \; \langle \xi, b_{[-1]} \rightharpoonup m_{-1} \rangle (b_{[0]} \ast m_0)
\\[.4pc] 
&\; = \; b_{[0]} \ast \big( \langle \xi \leftharpoonup b_{[-1]}, m_{-1} \rangle m_0\big) 
& \; = \; b_{[0]} \ast \big( (\xi \leftharpoonup b_{[-1]}) \star m \big), 
\end{array}
\]
}

\smallskip

\noindent for $\xi \in C^*, b \in B, m \in M$. Therefore, the actions $\ast$ (resp., $\star$) of $B$ and (resp., $(C^*)^{\textnormal{op}}$) on $M$ induce the stated action $\star$ of the crossed product $B \rtimes_L (C^*)^{\textnormal{op}}$ on $M$, that is, 
$M \in \lmod{(B \rtimes_L (C^*)^{\textnormal{op}})}$. Clearly, the space of morphisms between
$M, N \in {}^C_B  \mathsf{DH}(L)$ coincides with the space of morphisms between $M$ and $N$ considered as $B \rtimes_L (C^*)^{\textnormal{op}}$-modules. This yields a functor:
\[
\Omega_{L; B, C}: {}^C_B  \mathsf{DH}(L) \to \lmod{(B \rtimes_L (C^*)^{\textnormal{op}})}.
\]
Moreover, this is an isomorphism of categories since $\lcomod{C} \cong \lmod{(C^*)^{\textnormal{op}}}$ by Lemma~\ref{lem:C-comod}.
\end{proof}

\begin{remark} \label{rem:GM}
Note that, up to differences in conventions, the crossed product algebra $B \rtimes_L (C^*)^{\textnormal{op}}$ is a smash product algebra as defined in \cite{Takeuchi}, cf.~\cite{CMZ}*{Section 3}.
We also note that sometimes  Doi-Hopf modules are referred to as Doi--Koppinen modules due to independent work of Koppinen in \cite{Kop}.  
\end{remark}

\subsubsection{The right $H$-module algebra $\widehat{H}^*$} 
\label{sec:hatHstar}
Recall the left $H$-module coalgebra $\widehat{H}$ introduced in Section~\ref{sec:hatH}. 
Then, when $H$ is finite-dimensional, the dual vector space $\widehat{H}^*$ has a canonical structure of a right $H$-module algebra, which is described in the next lemma. 

\begin{lemma} \label{lem:hatHstar} 
Assume that $H$ is finite-dimensional. Then, given the left $H$-module coalgebra $\widehat{H}$ in \Cref{def:hatH}, we have the following statements.
\begin{enumerate}[\upshape (a)]
\item The induced algebra structure on $\widehat{H}^*$ is given as follows, for $\xi, \zeta \in \widehat{H}^*$:
\begin{align*}
\textnormal{mult}^{\widehat{H}^*}(\xi, \zeta)
& \; = \;
\textstyle \sum_{i,j} \big( t_i \twoheadrightarrow \xi \twoheadleftarrow S(t_j) \big)
\big( s_i s_j \twoheadrightarrow \zeta \big).
\end{align*}

\smallskip

\item The induced right $H$-module algebra structure of $\widehat{H}^*$ is given as follows, for $\ell \in H$, $\xi \in \widehat{H}^*$:
\[
\xi \leftharpoonup \ell = S^{-1}(\ell_{(1)}) \twoheadrightarrow \xi \twoheadleftarrow \ell_{(2)}.
\]  
\end{enumerate}
\end{lemma}

\begin{proof} Part (a) holds by the following computation:
{\small
\begin{align*}
\langle \xi \zeta, h \rangle \; &= \; 
\langle \xi  \otimes  \zeta, \; \widehat{\Delta}(h) \rangle
\\
\; &
\overset{\textnormal{Def.\;\ref{def:hatH}}}{=} \; 
\langle \xi \otimes  \zeta, \; \textstyle \sum_{i,j} t_j  h_{(1)}  t_i \; \otimes \; h_{(2)}  s_i  S^{-1}(s_j) \rangle
\\[.2pc]
\; &= \; \textstyle \sum_{i,j} \langle \xi, \; t_j  h_{(1)}  t_i \rangle \; \langle \zeta,\; h_{(2)}  s_i  S^{-1}(s_j) \rangle
\\[.2pc] 
\; & \overset{\textnormal{\eqref{eq:twohead}}}{=} \; \textstyle \sum_{i,j} \langle t_i \twoheadrightarrow \xi \twoheadleftarrow t_j, h_{(1)} \rangle \; \langle s_i S^{-1}(s_j) \twoheadrightarrow \zeta, \; h_{(2)} \rangle 
\\[.2pc] 
\; & \overset{\textnormal{\eqref{eq:Hopfpair}}}{=} \; \textstyle \sum_{i,j} \langle ( t_i \twoheadrightarrow \xi \twoheadleftarrow t_j ) ( s_i S^{-1}(s_j) \twoheadrightarrow \zeta), \; h \rangle
\\
\; & \overset{\textnormal{\eqref{eq:QT5}}}{=} \; \textstyle \sum_{i,j} \langle ( t_i \twoheadrightarrow \xi \twoheadleftarrow S(t_j) ) ( s_i s_j \twoheadrightarrow \zeta), \; h \rangle, 
\end{align*}
}

\noindent for all $\xi, \zeta \in \widehat{H}^*$ and $h \in \widehat{H}$. Part (b) is proved by the next computation:
{\small
\[
\langle \xi \leftharpoonup \ell, h \rangle \; = \; 
\langle \xi, \ell \rightharpoonup h \rangle
\;  \overset{\textnormal{Def.\;\ref{def:hatH}}}{=} \; \langle \xi, \;\ell_{(2)} h S^{-1}(\ell_{(1)}) \rangle 
 \; \overset{\textnormal{\eqref{eq:twohead}}}{=} \; \langle S^{-1}(\ell_{(1)}) \twoheadrightarrow \xi \twoheadleftarrow \ell_{(2)}, \; h \rangle,
\]
}

\noindent for all $\xi \in \widehat{H}^*, h \in \widehat{H}, \ell \in H$.
\end{proof}

%%%%%%%%%%%%%%%%%%%%%%%%%%%%%

\subsubsection{Definition of the reflective algebra}
\label{sec:refl-alg}

Now we present the main construction of this section.

\begin{definition}
For a finite-dimensional quasitriangular Hopf algebra $H$ and a left $H$-comodule algebra $A$, define the {\emph{reflective algebra of $A$ with respect to $H$}} to be the crossed product algebra:
\[
R_H(A):= A \rtimes_H (\widehat{H}^*)^{\textnormal{op}}.
\]
\end{definition}

The algebras $A$ and $(\widehat{H}^*)^{\textnormal{op}}$ are canonical subalgebras of the reflective algebra $R_H(A)$. Moreover, by Lemma~\ref{lem:cross-prod}, $R_H(A)$ is isomorphic to $A \otimes (\widehat{H}^*)^{\textnormal{op}}$ as an $(A, (\widehat{H}^*)^{\textnormal{op}})$-bimodule. Also, pertaining to Majid's transmuted Hopf algebras discussed in Remark~\ref{rem:Majid}, we have that 
\[
R_H(A) \; \cong \; A \rtimes_H (\widehat{H}^{\textnormal{cop}})^* \cong  A \rtimes_H (H^{\textnormal{cop}})_{\textnormal{trm}}, \quad \text{as $\Bbbk$-algebras}.
\]

Specializing the functor $\Omega_{L,B,C}$ from Proposition \ref{prop:isom-DH-double} to 
$L:=H, B:=A, C:=\widehat{H}^*$, gives the following corollary. 

\begin{corollary} \label{cor:isom-refl-alg} For a finite-dimensional quasitriangular Hopf algebra $H$ and a left $H$-comodule algebra $A$, 
the functor 
\[
\Omega:= \Omega_{H; A, \widehat{H}^*} : {}_A^{\widehat{H}}  \mathsf{DH}(H) \overset{\sim}{\longrightarrow} \lmod{R_H(A)}.
\]
is an isomorphism of  categories. \qed
\end{corollary}

%%%%%%%%%%%%%%%%%%%%%%%%%%

\subsection{Properties of reflective algebras} \label{sec:prop-ref}
In this part, we examine algebraic properties of reflective algebras and their categories of modules. 
%First, we recall some terminology about left $H$-comodule algebras $A$ below: 
%\begin{itemize}
%\item $A$ is a {\it left coideal subalgebra} of $H$ if it is both a left coideal (that is, a left $H$-comodule contained in $H$) and a subalgebra of $H$;
%\smallskip
%\item An {\it $H$-ideal} $I$ of $A$ is an ideal of $A$ that is also an left $H$-comodule;
%\smallskip
%\item $A$ is an {\it $H$-simple} left $H$-comodule algebra if $A$ has no non-trivial  $H$-ideal;
%\smallskip
%\item  $A$ is {\it semisimple} if it is semisimple as a $\Bbbk$-algebra in the usual sense: if either the left or right regular $A$-module is semisimple, or if the category $\lmod{A}$ is a semisimple category;
%\smallskip
%\item $A$ is {\it exact} if $\lmod{A}$ is an exact left $(\lmod{H})$-module category.  
%\end{itemize}

\begin{proposition} \label{prop:RHA-props}
For a finite-dimensional quasitriangular Hopf algebra $H$ and a left $H$-comodule algebra $A$, we have the following facts about the reflective algebra $R_H(A)$ and its category of modules. 
\begin{enumerate}[\upshape (a)]
%\item \stkout{$R_H(A)$-$\mathsf{mod}$ is abelian.

%\smallskip

\item If $A$ is finite-dimensional, then $R_H(A)$ is finite-dimensional, and hence, $R_H(A)\text{-}\mathsf{fdmod}$ is a finite abelian category.

\smallskip

\item If $H$ is semisimple, and $A$ is finite-dimensional and semisimple, then $R_H(A)$ is semisimple.
\end{enumerate}
\end{proposition}

\begin{proof} 
%\stkout{Part (a) is well-known. 
Part (a) follows from Lemma~\ref{lem:cross-prod}. 
Part (b) follows from Corollary~\ref{cor:ECMprops}(e).
\end{proof}

\begin{example} \label{ex:specialA}
Consider the special cases of left $H$-comodule algebras $A$ below.
Take $A = \Bbbk$ to be the trivial left coideal subalgebra of $H$ with $\delta(1_\Bbbk) = 1_H \otimes 1_\Bbbk$. Then, 
\[
R_H(\Bbbk) \cong (\widehat{H}^*)^{\textnormal{op}}. 
\]
Here, $\lmod{R_H(\Bbbk)}$ is abelian and finite. Moreover, $R_H(\Bbbk)$ is semisimple when $H$ is semisimple.
\end{example}

%%%%%%%%%%%%%%%%%%%%%%%%%%
%%%%%%%%%%%%%%%%%%%%%%%%%%
%%%%%%%%%%%%%%%%%%%%%%%%%%

\section{Modules over reflective algebras as braided module categories}
\label{sec:R-comod-RHA}
We maintain the setting and notation of Section~\ref{sec:setting-repr} here. The goal of this section is to upgrade the category isomorphisms of the previous section to isomorphisms of braided module categories.  Namely in Sections~\ref{sec:upgrade} and~\ref{sec:upgrade2}, we establish how \eqref{eq:ECM-DH} and \eqref{eq:ECM-DH-2}, respectively, can be extended to isomorphism of braided left $\cC$-module categories. We also obtain an $H$-comodule algebra structure and quantum $K$-matrix (i.e., quasitriangular structure) for the pertinent reflective algebra in Section~\ref{sec:upgrade2}. Next in Section~\ref{sec:univK}, we display a universal property for the reflective algebra of the trivial $H$-comodule algebra $\Bbbk$. Then in Section~\ref{sec:RA-ex}, we provide an explicit example of the results here for $H$ being the Drinfeld double of a finite group. 

\medskip

\noindent {\bf Standing notation.}  Along with the notation of Section~\ref{sec:setting-repr}, we collect some additional notation introduced in the previous section. 
\begin{itemize}
    \item ${}_{B}^{C}  \mathsf{DH}(L)$ is the category of $(L,B,C)$-Doi--Hopf modules from \Cref{sec:Doi--Hopf}, with $L$ a Hopf algebra,  with objects $(M, \ast, \varphi)$ for $(M, \ast)$ a left $B$-module, and $(M, \varphi)$ a left $C$-module.\medskip
    \item $\varphi(m):= m_{-1} \otimes m_0$ for $m \in M$.
    \medskip
    \item $\widehat{H}:=(\widehat{H}, \widehat{\Delta}, \widehat{\varepsilon}, \rightharpoonup)$ is the left $H$-module coalgebra from \Cref{def:hatH}; here $\widehat{\Delta}(h):= h_{\widehat{(1)}} \otimes   h_{\widehat{(2)}}$.

    \vspace{.05in}
    
    \item When $H$ is finite-dimensional, $\langle \hspace{0.02in}, \rangle$ is the algebra-coalgebra pairing between $\widehat{H}^*$, $\widehat{H}$. Here,
\begin{equation} \label{eq:hatHpair}
\langle \xi \zeta, h  \rangle = \langle \xi \otimes \zeta, \widehat{\Delta}(h) \rangle = \langle \xi, h_{\widehat{(1)}} \rangle  \langle \zeta, h_{\widehat{(2)}} \rangle,
\end{equation}
for $\xi, \zeta \in \widehat{H}^*$, $h \in \widehat{H}$.

\medskip

\item $R_H(A)$ is the reflective algebra from Section~\ref{sec:refl-alg}; it is equal to $A \otimes H^*$ as a vector space.
\end{itemize}

\subsection{Reflective centers are  Doi--Hopf modules as braided module categories} \label{sec:upgrade}

By \eqref{eq:ECM-DH}, we have the category isomorphism 
\[
\cE_{\lmod{H}}({\lmod{A}}) \; \cong \;{}_A^{\widehat{H}}  \mathsf{DH}(H).
\]
The goal of this subsection to describe explicitly the corresponding braided $(\lmod{H})$-module category structure of ${}_A^{\widehat{H}}  \mathsf{DH}(H)$. 
This will be akin to the isomorphism $\cZ(\lmod{H}) \; \overset{\otimes}{\cong} \; \lYD{H}$ of braided categories mentioned in \eqref{eq:Drin}.

\begin{lemma} \label{lem:ECM-DH-brmod}
The following statements hold. 
\begin{enumerate}[\upshape (a)]
    \item We have that  $\cE_{\lmod{H}}({\lmod{A}})$  is a left module category over $\lmod{H}$ as follows:
\[
\begin{array}{rl}
\act: \lmod{H} \;\times \;\cE_{\lmod{H}}({\lmod{A}}) &\longrightarrow \cE_{\lmod{H}}({\lmod{A}}) \\[.4pc]
\left( (Y, \cdot), \; (M, \ast, e^M) \right) & \mapsto 
\left( Y \otimes M, \; \widetilde{\ast}, \; e^{Y \otimes M}\right),
\end{array}
\]
for $a \; \widetilde{\ast} \; (y \otimes m) = (a_{[-1]} \cdot y) \otimes (a_{[0]} \ast m)$ with $a \in A$, $y \in Y$, $m \in M$. Here, $e^{Y \otimes M}$ is given by: 
\begin{align} \label{eq:ECM-Hmod}
e^{Y \otimes M}_X(x \otimes y \otimes m) =  \textstyle \sum_{i,j,k} (t_k g_j s_i \cdot x) \otimes (s_k t_i \cdot y) \otimes (p_j \ast m),
\end{align}
for some element $\sum_j g_j\otimes p_j\in H\otimes A$ independent of choice of $Y,M$.

\medskip

\item Further, the reflections $e^M$ equip $\cE_{\lmod{H}}({\lmod{A}})$ with the structure of a braided module category over $\lmod{H}$, where $$e^{\cE}_{X,(M,e^M)} := e^M_X,$$ 
for $X \in \lmod{H}$ and $(M,e^M) \in \cE_{\lmod{H}}({\lmod{A}})$. In particular,  $e_X^M$  is a braiding if and only if $K:=\sum_j g_j \otimes p_j$ is a $K$-matrix for $A$.
\end{enumerate}
\end{lemma}

\begin{proof}
Part (a) follows from Lemma~\ref{lem:ECM-modcat}. In particular, the formula for $\widetilde{\ast}$ is derived from Lemma~\ref{lem:braid-Kmatr}(a). Note that there exists an element $\sum_j g_j\otimes p_j \in H \otimes A$  satisfying
$e^M_X (x\otimes m)= \textstyle \sum_j (g_j\cdot x) \otimes (p_j\ast m)$ by Lemma~\ref{lem:braid-Kmatr}(c). Then the formula for $e^{Y\otimes M}$ holds as follows:
{\small
\begin{align*}
e^{Y\otimes M}_X(x \otimes y \otimes m)
& \overset{\textnormal{\eqref{eq:brmod1}}}{=} 
(c_{Y,X} \otimes \ide_M) (\ide_Y \otimes e_{X}^M) (c_{X,Y} \otimes \ide_M)(x \otimes y \otimes m)\\[.2pc]
& = 
\textstyle \sum_i (c_{Y,X} \otimes \ide_M) (\ide_Y \otimes e_{X}^M) ((t_i \cdot y) \otimes (s_i \cdot x) \otimes m)\\[.2pc]
& = 
\textstyle \sum_{i,j} (c_{Y,X} \otimes \ide_M) ((t_i \cdot y) \otimes (g_js_i \cdot x) \otimes (p_j \ast m))\\[.2pc]
& = 
\textstyle \sum_{i,j,k} (t_kg_js_i \cdot x) \otimes (s_kt_i \cdot y)  \otimes (p_j \ast m).
\end{align*}
}

Part (b) holds by Proposition~\ref{prop:ECM-brmodcat}, and Lemma~\ref{lem:braid-Kmatr}(b).
\end{proof}

%See "Sec6-additions-CW20330623"

Next, we use the braided module category structure of $\cE_{\lmod{H}}({\lmod{A}})$ in the lemma above to induce such a structure for ${}_A^{\widehat{H}}  \mathsf{DH}(H)$.

\begin{proposition} \label{prop:ECM-DH-brmod}
We have that
\[
\cE_{\lmod{H}}({\lmod{A}}) \; \overset{\textnormal{br.mod}}{\cong} \;{}_A^{\widehat{H}}  \mathsf{DH}(H),
\]
as braided left $({\lmod{H}})$-module categories, 
where:
\begin{enumerate}[\upshape (a)]
\item The left $({\lmod{H}})$-module category structure on ${}_A^{\widehat{H}}  \mathsf{DH}(H)$ is given by 
\[
\begin{array}{rl}
\blkact \;: \lmod{H} \;\times \;{}_A^{\widehat{H}}  \mathsf{DH}(H) &\longrightarrow {}_A^{\widehat{H}}  \mathsf{DH}(H) \\[.4pc]
\left( (Y, \cdot), \; (M, \ast, \varphi) \right) & \mapsto 
\left( Y \otimes M, \; \widetilde{\ast}, \; \widetilde{\varphi}\right),
\end{array}
\]
for $a \; \widetilde{\ast} \; (y \otimes m) = (a_{[-1]} \cdot y) \otimes (a_{[0]} \ast m)$ with $a \in A$, $y \in Y$, $m \in M$, and  for $\varphi(m):=m_{-1} \otimes m_0$,
\begin{align} \label{eq:DH-Hmod}
\widetilde{\varphi}(y \otimes m) = \textstyle \sum_{i,j} (t_j m_{-1} s_i) \otimes (s_j t_i \cdot y) \otimes m_0.
\end{align}

\smallskip

\item The braiding on ${}_A^{\widehat{H}}  \mathsf{DH}(H)$ is given by 
$$e^{\mathsf{DH}}_{X,(M, \ast, \varphi)}(x \otimes m) := (m_{-1} \cdot x) \otimes m_0,$$ 
for $X \in \lmod{H}$ and $(M, \ast, \varphi) \in {}_A^{\widehat{H}}  \mathsf{DH}(H)$, with $x \in X, m \in M$.
\end{enumerate}
\end{proposition}

\begin{proof}
(a) By Proposition~\ref{prop:structuretransfer}(a), the action $\blkact$ is induced by the action $\act$ from Lemma~\ref{lem:ECM-DH-brmod}(a), the functor $F$ from Proposition~\ref{prop:functorF}, and its inverse $G$ from  Proposition~\ref{prop:functorG} as follows:
\[ 
\blkact \;: \lmod{H} \times {}_A^{\widehat{H}}  \mathsf{DH}(H) \overset{\textnormal{Id} \times G}{\longrightarrow} \lmod{H} \times \cE_{\lmod{H}}({\lmod{A}})
\overset{\act}{\longrightarrow}
\cE_{\lmod{H}}({\lmod{A}})
\overset{F}{\longrightarrow}
{}_A^{\widehat{H}}  \mathsf{DH}(H).
\]
The formula for $\widetilde{\ast}$ then follows  from Lemma~\ref{lem:ECM-DH-brmod}. Moreover, the formula for $\widetilde{\varphi}$  follows from the computations below:
{\small
\[
\begin{array}{ll}
\textstyle \sum_j g_j \otimes (p_j \ast m) \; = \; e_H^M(1_H \otimes m) \overset{\textnormal{Prop.\;\ref{prop:functorG}}}{=} m_{-1} \otimes  m_0 \; &\quad \text{(for $\textnormal{Id} \times G$ applied to $\varphi$)},\\[.6pc]
 \Rightarrow \; \; e_H^{Y \otimes M}(1_H \otimes y \otimes m) \; \overset{\textnormal{\eqref{eq:ECM-Hmod}}}{=} \; \textstyle \sum_{i,j} t_j m_{-1} s_i \otimes (s_j t_i \cdot y) \otimes m_0 \; &\quad \text{(then applying $\act$)},\\[.6pc]
 \therefore \; \; \widetilde{\varphi}(y \otimes m) \; \overset{\textnormal{Prop.\;\ref{prop:functorF}}}{=} \;  e_H^{Y \otimes M}(1_H \otimes y \otimes m)  \; &\quad \text{(finally applying $F$)}.
\end{array}
\]
}
%See "Sec6-additions-CW20330623"

\smallskip

(b) This follows from Proposition~\ref{prop:structuretransfer}(b), Lemma~\ref{lem:ECM-DH-brmod}(b), and computations as in part (a).
\end{proof}

%%%%%%%%%%%%%%%%%%%%%%%%%%

\subsection{Reflective algebras as $H$-comodule algebras with quantum $K$-matrices} \label{sec:upgrade2}
By \eqref{eq:ECM-DH-2}, we have the category isomorphism
\[
\cE_{\lmod{H}}({\lmod{A}})
\; \cong \; \lmod{R_H(A)}.
\]
The goal of this subsection is to describe explicitly the corresponding braided $\lmod{H}$ module category structure of 
$\lmod{R_H(A)}$. This will be akin to the isomorphism of braided categories, $\cZ(\lmod{H}) \; \overset{\otimes}{\cong} \; \lmod{\Drin(H)}$ mentioned in \eqref{eq:Drin}.

\begin{theorem} \label{thm:ECM-DH-brmod-2}
For a finite-dimensional quasitriangular Hopf algebra $H$ and a left $H$-comodule algebra $A$, we have that
\begin{equation}
\label{eq:br-isom-E-RHA}
\cE_{\lmod{H}}({\lmod{A}}) \; \overset{\textnormal{br.mod}}{\cong} \; 
\lmod{R_H(A)},
\end{equation}
as braided left $({\lmod{H}})$-module categories, 
where:
\begin{enumerate}[\upshape (a)]
\item The left $({\lmod{H}})$-module category structure on $\lmod{R_H(A)}$ is given by 
\[
\begin{array}{rl}
\grayact : \lmod{H} \;\times \;\lmod{R_H(A)}&\longrightarrow \lmod{R_H(A)} \\[.4pc]
\left( (Y, \cdot), \; (M, \ast, \star) \right) & \mapsto 
\left( Y \otimes M, \; \widetilde{\ast}, \; \widetilde{\star}\right),
\end{array}
\]
for $a \; \widetilde{\ast} \; (y \otimes m) = (a_{[-1]} \cdot y) \otimes (a_{[0]} \ast m)$ with $a \in A$, $y \in Y$, $m \in M$. Also  for $\xi \in (\widehat{H}^*)^{\textnormal{op}}$:
\begin{equation} \label{eq:RHA-Hmod}
\xi \; \widetilde{\star}\; (y \otimes m) = \textstyle \sum_{i,j,d} \langle \xi,\; t_j h_d s_i \rangle (s_j t_i \cdot y) \otimes (\xi_d \star m).
\end{equation}
Here, $\{h_d, \; \xi_d\}_d$ is a dual basis of $H$.
\smallskip

\item The braiding on $\lmod{R_H(A)}$ is given 
by 
$$e^{R_H}_{X,(M, \star)}(x \otimes m) := \textstyle \sum_d (h_d \cdot x) \otimes (\xi_d \star m),$$ 
for $X \in \lmod{H}$ and $(M, \star) \in \lmod{R_H(A)}$, with $x \in X, m \in M$.
\end{enumerate}
\end{theorem}

\begin{proof}
The braided isomorphism between $\cE_{\lmod{H}}({\lmod{A}})$ and  ${}_A^{\widehat{H}}  \mathsf{DH}(H)$ follows from Proposition~\ref{prop:ECM-DH-brmod}. To establish the braided isomorphism between ${}_A^{\widehat{H}} \mathsf{DH}(H)$ and $\lmod{R_H(A)}$, see the work below.

\smallskip

(a) According to Proposition~\ref{prop:structuretransfer}(a), the action $\grayact$ is induced by the action $\blkact$ from Proposition~\ref{prop:ECM-DH-brmod}(a), the functors $\Omega$ and its inverse from Corollary~\ref{cor:isom-refl-alg} as such:
\[ 
\grayact: \lmod{H} \times \lmod{R_H(A)}\overset{\textnormal{Id} \times \Omega^{-1}}{\longrightarrow} \lmod{H} \times {}_A^{\widehat{H}}  \mathsf{DH}(H) 
\overset{\blkact}{\longrightarrow}
{}_A^{\widehat{H}}  \mathsf{DH}(H) 
\overset{\Omega}{\longrightarrow}
\lmod{R_H(A)}.
\]
The formula for $\widetilde{\ast}$ follows  from Lemma~\ref{lem:ECM-DH-brmod}. The formula for $\widetilde{\star}:= \widetilde{\star}_{\widetilde{\varphi}_\star}$  from the computations below:
{\small
\[
\begin{array}{ll}
m_{-1} \otimes m_0 \; = \; \varphi(m) \; \overset{\textnormal{Lem.\;\ref{lem:C-comod}}}{=} \; \textstyle \sum_d h_d \otimes (\xi_d \star m)  \; &\quad \text{(for $\textnormal{Id} \otimes \Omega^{-1}$ applied to $\star$)},\\[.6pc]
 \Rightarrow \; \; \widetilde{\varphi}_\star(y \otimes m) \; \overset{\textnormal{\eqref{eq:DH-Hmod}}}{=} \; \textstyle \sum_{i,j,d} (t_j h_d s_i) \otimes (s_j t_i \cdot y) \otimes (\xi_d \star m) \; &\quad \text{(then applying $\blkact$)},\\[.6pc]
 \therefore \; \; \xi \; \widetilde{\star} \; (y \otimes m) \; \overset{\textnormal{Lem.\;\ref{lem:C-comod}}}{=} \; \textstyle \sum_{i,j,d} \langle \xi, t_j h_d s_i \rangle (s_j t_i \cdot y) \otimes (\xi_d \star m)  \; &\quad \text{(finally applying $\Omega$)}.
\end{array}
\]
}
%See "Sec6-additions-CW20330623"

\smallskip

(b) This follows from Propositions~\ref{prop:structuretransfer}(b) and~\ref{prop:ECM-DH-brmod}(b), and computations as in part (a).
\end{proof}

Now we obtain a quasitriangular structure for the reflective algebra $R_H(A)$.

\begin{corollary} 
\label{cor:RHA-comod}
For a finite-dimensional quasitriangular Hopf algebra $H$ and a left $H$-comodule algebra $A$, we have the statements below about the reflective algebra $R_H(A)$.
\begin{enumerate}[\upshape (a)]
\item  $R_H(A)$ is a left $H$-comodule algebra with left $H$-coaction $\delta_{\textnormal{ref}}$ on $R_H(A)$ given by:
\begin{align*}
\delta_{\textnormal{ref}}(a) &:= a_{[-1]} \otimes a_{[0]},\\[.3pc]
\delta_{\textnormal{ref}}(\xi) &:= \textstyle \sum_{i,j,d} \langle \xi,\;  t_j h_d s_i \rangle  s_j t_i \; \otimes \; \xi_d \quad (=: \xi_{[-1]} \otimes \xi_{[0]}),\\[.3pc]
\delta_{\textnormal{ref}}(a \; \xi)&:= \delta_{\textnormal{ref}}(a)\; 
\delta_{\textnormal{ref}}(\xi) \quad (=: a_{[-1]} \xi_{[-1]} \otimes a_{[0]} \xi_{[0]}), 
\end{align*}
for $a \in A$ and $\xi \in (\widehat{H}^*)^{\textnormal{op}}$. Here, recall that $\sum_i s_i \otimes t_i$ is the $R$-matrix of $H$, and that $\{h_d, \xi_d\}_d$ is a dual basis of $H$.

\medskip

\item $R_H(A)$ is quasitriangular (as an $H$-comodule algebra) with $K$-matrix:
\begin{equation}
\label{eq:K-A}
K_{\textnormal{ref}}(A) := \textstyle \sum_d h_d \otimes \xi_d \; \in H \otimes (\widehat{H}^*)^{\textnormal{op}} \; \subset H \otimes R_H(A). 
\end{equation}
%Here, $K_{\textnormal{ref}}(A)$ does not depend on the choice of dual bases of $H$. 
\end{enumerate}
\end{corollary}

\begin{proof}
(a) Since \eqref{eq:br-isom-E-RHA} is an isomorphism of 
left $\lmod{H}$-module categories, 
%The braided It follows from Lemma \ref{lem:cross-prod} that 
\begin{equation} \label{eq:delta-ref}
r \cdot (x \otimes m) = \delta_{\textnormal{ref}}(r) (x \otimes m),
\end{equation}
for all $r \in R_H(A)$, $x \in X$, $m \in M$
for some $\delta_{\textnormal{ref}} : R_H(A) \to H \otimes R_H(A)$ comodule algebra map. The formulas for $\delta_{\textnormal{ref}}(a)$ and $\delta_{\textnormal{ref}}(\xi)$ follow from the formulas for $\widetilde{\ast}$ and $\widetilde{\star}$, respectively, in Theorem~\ref{thm:ECM-DH-brmod-2}(a), 
applied to the left regular modules $X=H, M=A$ and the elements $x=1_H, m=1_A$.

(b) This follows from Lemma~\ref{lem:braid-Kmatr}(b) and Theorem~\ref{thm:ECM-DH-brmod-2}(b).
In the ArXiv version 1 of this article it is shown directly that
$\delta_{\textnormal{ref}}$ defines an $H$-comodule algebra structure and that the axioms \eqref{eq:K2}-\eqref{eq:K1} hold for $K_{\textnormal{ref}}(A)$.
\end{proof}

\begin{example} 
\label{ex:specialA2}
Corollary \ref{cor:RHA-comod} implies that the isomorphism $R_H(\Bbbk) \cong (\widehat{H}^*)^{\textnormal{op}}$ from Example \ref{ex:specialA} 
is an isomorphism of $H$-comodule algebras.
\end{example}

\begin{example} \label{ex:cocom}
When $H$ is cocommutative, it is quasitriangular with $R = 1_H \otimes 1_H$. Here, $R_H(A)$ is a left $H$-comodule algebra, where for $a \in A$ and  $\xi \in (\widehat{H}^*)^{\textnormal{op}}$, we have $\delta_{\textnormal{ref}}(a) = a_{[-1]} \otimes a_{[0]}$ (identified with $a_{[-1]} \otimes (a_{[0]} \otimes \widehat{\varepsilon})$ in $H \otimes R_H(A)$)
and
\[\delta_{\textnormal{ref}}(\xi) = \textstyle \sum_{i,j,d} \langle \xi, h_d  \rangle  \otimes \; \xi_d \overset{\textnormal{\eqref{eq:Hopfpair}}}{=} \xi
\qquad \text{(identified with $1_H \otimes (1_A \otimes \xi)$ in $H \otimes R_H(A)$)}.\]
\end{example}

%%%%%%%
% See version 20230627 for the commented out material compiled.
%%%%%%%

\subsection{Universality of the reflective algebra $R_H(\Bbbk)$}
\label{sec:univK}

In this part, we show that the reflective algebra $R_H(\Bbbk)$ arises as an initial object of the category of quasitriangular left $H$-comodule algebras. Here, we assume that $H$ is finite-dimensional.

\begin{definition} \label{def:H-QT} Let ${}^H \mathsf{QT}$
be the category of quasitriangular left $H$-comodule algebras. Namely, 
\begin{enumerate}[(a)]
\item Objects are pairs, $(Q, K)$, where $Q$ is a left $H$-comodule algebra, 
and $K:=K(Q) \in H \otimes Q$ is a quantum $K$-matrix for $Q$, and 
\smallskip
\item A morphism from $(Q_1,K_1)$ to $(Q_2,K_2)$ is a linear map $\phi: Q_1 \to Q_2$ that is both a  left $H$-comodule morphism and an algebra morphism, such that $K_2 = (\ide_H \otimes \phi)(K_1)$.
\end{enumerate}
\end{definition}

Indeed, if $(Q,K) \in {}^H \mathsf{QT}$, then the identity morphism $\ide_{(Q,K)}$ is equal to $\ide_Q$ because $\ide_Q$ is a left $H$-comodule algebra morphism and $K = (\ide_H \otimes \ide_Q)(K)$. Also if we have morphisms $\phi_1:(Q_1,K_1) \to (Q_2,K_2)$ and $\phi_2: (Q_2,K_2) \to (Q_3,K_3)$ in  ${}^H \mathsf{QT}$, then $\phi_2 \phi_1: (Q_1,K_1) \to (Q_3,K_3)$ is in  ${}^H \mathsf{QT}$ since it is a left $H$-comodule algebra morphism and $(\ide_H \otimes \phi_2 \phi_1)(K_1) = (\ide_H \otimes \phi_2)(K_2) = K_3$.

\medskip

Examples of objects of ${}^H \mathsf{QT}$ include the pairs $(R_H(A), K_{\textnormal{ref}}(A))$, for the reflective algebra $R_H(A)$ of $A$ from Section~\ref{sec:refl-alg}, with $K$-matrix $K_{\textnormal{ref}}(A)$ given in \eqref{eq:K-A}. 

\medskip

Now the main result of this section is given below. 

\begin{theorem}
\label{thm:init-obj}
When $H$ is a finite-dimensional quasitriangular Hopf algebra over $\Bbbk$, we have that $(R_H(\Bbbk), K_{\textnormal{ref}}(\Bbbk))$ is an initial object of ${}^H \mathsf{QT}$.
\end{theorem}

\begin{proof}
For an arbitrary object $(Q, K) \in {}^H \mathsf{QT}$, our task is to produce a unique morphism 
$$\kappa:= \kappa_{(Q,K)}: (R_H(\Bbbk), K_{\textnormal{ref}}(\Bbbk)) \to (Q, K)$$ in ${}^H \mathsf{QT}$. Towards this, recall that $R_H(\Bbbk) \cong (\widehat{H}^*)^{\textnormal{op}}$ as left $H$-comodule algebras [Example~\ref{ex:specialA2}]. So, $R_H(\Bbbk) \cong H^*$ as vector spaces [\Cref{def:hatH}]. Now we use the element $K:= \sum_i g_i \otimes p_i \in H \otimes Q$ to yield a linear map:
\begin{equation} \label{eq:kappa}
\kappa:= \kappa_{(Q,K)}: R_H(\Bbbk) \cong H^* \to Q, \quad \xi \mapsto \textstyle \sum_i \langle \xi, g_i \rangle \; p_i.
\end{equation}
It now remains to verify the following conditions for the linear map in \eqref{eq:kappa}.
\begin{enumerate}[(i)]
\item It is a left $H$-comodule morphism.
\smallskip

\item It is an algebra morphism; 
\smallskip

\item $K = (\ide_H \otimes \kappa)(K_{\textnormal{ref}}(\Bbbk))$.
\smallskip

\item Uniqueness: If $\kappa': H^* \to Q$ is a linear map satisfying (i)--(iii), then $\kappa' = \kappa$.
\end{enumerate}

\smallskip

Towards (i), we have that the left $H$-comodule structure on $R_H(\Bbbk) \cong (\widehat{H}^*)^{\textnormal{op}}$ is given by 
\[
\delta_{\textnormal{ref}}(\xi) = \textstyle \sum_{i,j,d} \langle \xi,\;  t_j h_d s_i \rangle  s_j t_i \; \otimes \; \xi_d 
\]
for $\xi \in (\widehat{H}^*)^{\textnormal{op}}$ by Corollary \ref{cor:RHA-comod}(a). Here, $\sum_i s_i \otimes t_i$ is the $R$-matrix of $H$, and $\{h_d, \xi_d\}_d$ is a dual basis of $H$. Now \eqref{eq:kappa} is a left $H$-comodule morphism as shown below:
{\small
\begin{align*}
\delta_Q(\kappa(\xi)) \; &= \; 
\textstyle \sum_i \langle \xi, g_i \rangle \; (p_i)_{[-1]}  \otimes  (p_i)_{[0]} 
  \overset{\textnormal{\eqref{eq:K3}}}{=} \; \textstyle \sum_{j,k,\ell} \langle \xi, t_j g_k s_l \rangle \; s_j t_l  \otimes p_k
\\
\; & \overset{\textnormal{\eqref{eq:dual-iden}}}{=} \textstyle \sum_{i,j,k,d}  \langle \xi,\;  t_j h_d s_i \rangle\;  \langle \xi_d,\;  g_k \rangle \;  s_j t_i \otimes p_k 
 = \; 
(\ide_H \otimes \kappa) ( \delta _{\textnormal{ref}}(\xi)). 
\end{align*}
}

To verify (ii), take $\xi, \zeta \in (\widehat{H}^*)^{\textnormal{op}}$ and consider the computation below:
{\small
\begin{align*}
\kappa(\xi \zeta)
\; &=\textstyle \sum_{i} \langle \xi \zeta, g_i \rangle \;p_i  \\[.2pc]
\; & \overset{\textnormal{\eqref{eq:hatHpair}}, \; \textnormal{Def.~\ref{def:hatH}}}{=} \textstyle \sum_{i,j',k'} \langle \zeta, t_{k'}  (g_i)_{(1)}  t_{j'} \rangle \; \langle \xi, \; (g_i)_{(2)} s_{j'}  S^{-1}(s_{k'}) \rangle \;p_i\\
\; &\overset{\textnormal{\eqref{eq:K2}}}{=} \textstyle \sum_{j,k,l,m,j',k'} \langle \zeta, \; t_{k'}  t_k  g_l  t^m  t_{j'} \rangle \; \langle \xi, \; g_j  s_k  s^m  s_{j'}  S^{-1}(s_{k'}) \rangle \;p_j p_l
\\[.2pc]
\; &= \textstyle \sum_{j,k,l,k'} \langle \zeta, \; t_{k'}  t_k  g_l \rangle \; \langle \xi, \; g_j S^{-1}(s_{k'} S(s_k)) \rangle \;p_j p_l\\
\; & \overset{\textnormal{\eqref{eq:QT5}}}{=} \textstyle \sum_{j,l} \langle \zeta,  g_l \rangle \; \langle \xi,  g_j) \rangle \;p_j p_l\\[.2pc]
\; &= \kappa(\xi) \kappa(\zeta).
\end{align*}
}

Next, we establish (iii) as follows:
{\small
\begin{align*}
(\ide_H \otimes \kappa)(K_{\textnormal{ref}}(\Bbbk))  \; \overset{\textnormal{\eqref{eq:K-A}}}{=} \; \textstyle \sum_{i,d} \langle \xi_d, g_i \rangle \;  h_d \otimes p_i \; \overset{\textnormal{\eqref{eq:dual-iden}}}{=} \; \textstyle \sum_{i}  g_i \otimes p_i = K.
\end{align*}
}

Finally, we verify (iv). If $K = \textstyle (\ide_H \otimes \kappa')(K_{\textnormal{ref}}(\Bbbk))$, then $\textstyle \sum_i g_i \otimes p_i = \sum_d h_d \otimes \kappa'(\xi_d)$. Applying $\langle \xi_d, - \rangle$ to the first factor yields $\kappa'(\xi_d) = \textstyle \sum_i \langle \xi_d, g_i \rangle \; p_i$, for all $d$. Since $\{\xi_d\}_d$ is a basis for $H^*$, we get that $\kappa'(\xi) = \textstyle \sum_i \langle \xi, g_i \rangle \; p_i$ for all $\xi \in H^*$ as desired.
\end{proof}

One may compare the result above to \cite{Radford1994}*{Theorem~1} on $\Drin(H)$ realized as a universal {\it quasitriangular envelope} of $H$. Moreover, the verification of (ii) in the proof above compares to \cite{BBJ}*{Theorem~4.9} with the distinction that we work with $H$-comodule algebras $A$. Next, consider the following example.

\begin{example}
\label{rem:K-matr-1dim-A-to-gen} 
For the object $(R_H(A), K_{\textnormal{ref}}(A)) \in {}^H \mathsf{QT}$, we obtain via Lemma~\ref{lem:cross-prod}  a canonical algebra embedding 
\[
\iota_A : R_H(\Bbbk) \cong (\widehat{H}^*)^{\textnormal{op}} \hookrightarrow R_H(A).
\]
By Corollary~\ref{cor:RHA-comod}(a), this is an embedding of $H$-comodule algebras, and from~\eqref{eq:K-A} we have \linebreak 
$\iota_A(K_{\textnormal{ref}}(\Bbbk)) = K_{\textnormal{ref}}(A). $
Therefore, $\iota_A$ is the unique homomorphism
$\kappa_{(R_H(A), K_{\textnormal{ref}}(A))}$
from the proof of Theorem \ref{thm:init-obj}.   
\end{example}

%%%%%%%%%%%%%%%%

\subsection{Example for the Drinfeld double of a finite group}  \label{sec:RA-ex}
In this subsection, we illustrate how our results apply to the case when $H$ is the Drinfeld double of a finite group. 

Take $G$ to be a finite group. Let
$\Bbbk G$ be the group algebra on $G$, and consider its Hopf dual, $(\Bbbk  G)^*$, the algebra of functions on $G$. 
Denote by $\{x\}_{x \in G}$ and 
$\{ \delta_x \}_{x \in G}$ the standard $\Bbbk$-bases of $\Bbbk G$ and $(\Bbbk  G)^*$, respectively. Also, take $\delta_{g,h}$ to be the Kronecker delta function,  for $g,h \in G$.

The Drinfeld double $\Drin(G) := \Drin(\Bbbk G)$ contains $\Bbbk G$ and $((\Bbbk G)^*)^{\textnormal{op}}$ as Hopf subalgebras, and is  $(\Bbbk G)^* \otimes  \Bbbk G$ as a $\Bbbk$-coalgebra. The $\Bbbk$-basis of $\Drin(G)$ is given by $\{ \delta_x y \}_{x, y \in G}$, with product:
\[
\big(\delta_x y \big) \big( \delta_{x'} y' \big) 
= \delta_{x, y x' y^{-1}} \;  \delta_x \; y y',  
\]
for $x, x', y, y' \in G$.  
The  quantum $R$-matrix of $\Drin(G)$ is
\begin{equation}
\label{eq:R-D(G)}
R = \textstyle \sum_{g \in G}  \delta_g \otimes g. 
\end{equation}
Also, as $\Bbbk$-algebras,  $(\Drin(G))^*$ is isomorphic to
$\Bbbk  G \otimes (\Bbbk G)^*$, with $\Bbbk$-basis $\{x \delta_y \}_{x, y \in G}$ and product:
\[
(x \delta_y) (x' \delta_{y'}) = 
\delta_{y,y'} \; x x' \; \delta_y.
\]

Now we illustrate Corollary~\ref{cor:RHA-comod} (and Theorem~\ref{thm:init-obj})  for the left $\Drin(G)$-comodule algebra $\Bbbk$. 

\begin{proposition}
Retain the notation above, and take $H:=\Drin(G)$. Then we obtain that the reflection algebra $R_H(\Bbbk) \cong (\widehat{H}^*)^{\textnormal{op}}$ is a quasitriangular left $H$-comodule algebra as follows:
\begin{enumerate}[\upshape (a)]
\item Its algebra structure  is given as follows, for all $x, y, x', y' \in G$:
\[
\big( x \delta_y \big) \big( x' \delta_{y'} \big)= 
\delta_{y', y^{-1} x^{-1} y x y} \; y^{-1} x y x' y^{-1} x^{-1} y x \; \delta_y.
\]

\smallskip
\item Its left $H$-comodule structure 
is given as follows, for all $x, y \in G$:
\[
\textstyle \delta_{\textnormal{ref}} (x \delta_y) = 
\textstyle \sum_{g \in G} 
\delta_g\;  y^{-1} x y \; \otimes\;  g^{-1} x g \; \delta_{g^{-1} y} \quad \in H \otimes H^* \subset H \otimes R_H(A).
\]

\smallskip
\item Its quantum $K$-matrix is given by $K=\sum_{g,h \in G}  \delta_g h \otimes g \delta_h \in H \otimes R_H(A)$.
\end{enumerate}
\end{proposition}

\begin{proof} (a)  To start, we compute:
\[
\langle g \twoheadrightarrow x \delta_y \twoheadleftarrow h, 
\delta_u v \rangle = 
\langle x \delta_y, h \delta_u v g \rangle 
= \langle x \delta_y, \delta_{h u h^{-1}} h v g \rangle = 
\delta_{u, h^{-1} x h} \; \delta_{v, h^{-1} y g^{-1}}
\]
for all $x,y,g,h,u,v \in G$. Thus, for $x,y, g, h \in G$: 
\[
g \twoheadrightarrow x \delta_y \twoheadleftarrow h
= h^{-1} x h \; \delta_{h^{-1} y g^{-1}}.
\]
Analogously, one obtains, for $x,y, g \in G$: 
\[
\delta_g \twoheadrightarrow x \delta_y 
= \delta_{g, y^{-1} x y} \; x \delta_y.
\]
Lemma \ref{lem:hatHstar}(a) and \eqref{eq:R-D(G)} now imply that the product structure of $(\widehat{H}^*)^{\textnormal{op}}$ is given by
{\small
\begin{align*}
(x \delta_y) (x' \delta_{y'}) &= 
\textstyle \sum_{g, h \in G} (g \twoheadrightarrow x' \delta_{y'} \twoheadleftarrow  h^{-1} ) ( \delta_g \delta_{h} \twoheadrightarrow x \delta_y ) 
= 
\textstyle \sum_{g \in G} (g \twoheadrightarrow x' \delta_{y'} \twoheadleftarrow g^{-1} ) ( \delta_g \twoheadrightarrow x \delta_y ) 
\\[.2pc]
&= \textstyle \sum_{g \in G} 
\delta_{g, y^{-1} x y}\; \delta_{g y' g^{-1},y}
\; g x' g^{-1} x \;\delta_y,
\end{align*}
}

\noindent which yields the statement in part (a) after simplifying the right-hand side. 

\smallskip

(b) Note that $\{\delta_k k', k \delta_{k'} \}_{k,k' \in G}$ is a dual basis of $H$.
Corollary \ref{cor:RHA-comod}(a) and \eqref{eq:R-D(G)} imply
{\small
\begin{align*}
\delta_{\textnormal{ref}} (x \delta_y) &= 
\textstyle \sum_{g,h,k,k' \in G} \; \langle x \delta_y,\;  g \delta_k k' \delta_h \rangle \;  \delta_g h \otimes k \delta_{k'}  
\\[.2pc]
&= 
\textstyle \sum_{g,h,k,k' \in G}\;  \langle x \delta_y, \; \delta_{g k g^{-1}} \; \delta_{g k' h (k')^{-1} g^{-1}}\;  g k' \rangle \; \delta_g h \otimes k \delta_{k'} 
\\[.2pc]
&= 
\textstyle \sum_{g,h,k,k' \in G} \;
\delta_{x, g k g^{-1}}\;
\delta_{x, g k' h (k')^{-1} g^{-1}}\;
\delta_{y, gk'} \; \delta_g h \otimes k \delta_{k'}
\\[.2pc]
&= 
\textstyle \sum_{g \in G} 
\delta_g\;   y^{-1} x y \; \otimes\; g^{-1} x g \; \delta_{g^{-1} y},
\end{align*}
}

\noindent where on the fourth line, the delta functions imply 
$k' = g^{-1} y, k = g^{-1} x g$ and 
$h = y^{-1} x y$.
The statement in part (b) is then obtained after simplifying the right-hand side. 

\smallskip

Part (c) follows directly from Corollary~\ref{cor:RHA-comod}(b).
\end{proof}

\subsection{Reflective algebras as comodule algebras over Drinfeld doubles} 
We return to the general standing notation from Section~\ref{sec:setting-repr}. 

The action of $\cZ(\lmod{H})$ on 
$\cE_{\lmod{H}}(\lmod{A})$ from Corollary~\ref{cor:ECMprops}(a) can be explicitly computed. Based on this, one obtains explicit actions of the category of Yetter--Drinfeld modules 
$\lYD{H}$ on the category of Doi--Hopf modules
${}_A^{\widehat{H}} \mathsf{DH}(H)$ and, when $H$ is a finite dimensional Hopf algebra, of 
$\lmod{\Drin(H)}$ on $\lmod{R_H(A)}$. 
The latter result equips $R_H(A)$ with a left $\Drin(H)$-comodule algebra structure. These results and their detailed proofs appear in Section 7 of the ArXiv version 1 of this article. Here, we only record the last result:

\begin{proposition} \label{prop:Drin-RHA-comod}
If $H$ is a finite dimensional Hopf algebra and $A$ a left $H$-comodule algebra, then the reflective algebra $R_H(A)$ is a left $\Drin(H)$-comodule algebra  as follows:
\begin{align*}
\delta^{\textnormal{{\tiny Drin}}}_{\textnormal{ref}}(a) &:= a_{[-1]} \otimes a_{[0]},\\[.3pc]
\delta^{\textnormal{{\tiny Drin}}}_{\textnormal{ref}}(\xi) &:= \textstyle \sum_{k,d,d'} \langle \xi, t_k h_{d'} S^{-1}(h_d) \rangle \; \left((s_k)_{(3)} \twoheadrightarrow \xi_d \twoheadleftarrow S((s_k)_{(1)})\right)(s_k)_{(2)} 
\; \otimes \;\xi_{d'},\\[.3pc]
\delta^{\textnormal{{\tiny Drin}}}_{\textnormal{ref}}(a \; \xi)&:= \delta_{\textnormal{ref}}(a)\; 
\delta^{\textnormal{{\tiny Drin}}}_{\textnormal{ref}}(\xi), 
\end{align*}
for $a \in A$ and $\xi \in (\widehat{H}^*)^{\textnormal{op}}$. Here, recall that $\sum_k s_k \otimes t_k$ is the $R$-matrix of $H$, and that $\{h_d, \xi_d\}_d$,  $\{h_{d'}, \xi_{d'}\}_{d'}$ are dual bases of $H$. \qed
\end{proposition}

\begin{example} \label{ex:cocom-Drin}
Let us continue Example~\ref{ex:cocom} when $H$ is cocommutative with $R = 1_H \otimes 1_H$. Here, $R_H(A)$ is a left $\Drin(H)$-comodule algebra, where for $a \in A$ and  $\xi \in (\widehat{H}^*)^{\textnormal{op}}$, we have $\delta^{\textnormal{{\tiny Drin}}}_{\textnormal{ref}}(a) = a_{[-1]} \otimes a_{[0]}$ (identified with $(\varepsilon \otimes a_{[-1]}) \otimes (a_{[0]} \otimes \widehat{\varepsilon})$ in $\Drin(H) \otimes R_H(A)$)
and
\[\delta^{\textnormal{{\tiny Drin}}}_{\textnormal{ref}}(\xi) = \textstyle \sum_{d,d'} \langle \xi, h_{d'} S^{-1}(h_d) \rangle \left(1_H \twoheadrightarrow \xi_d \twoheadleftarrow 1_H\right) \; 
 \otimes \;\xi_{d'} 
 = \textstyle \sum_{d,d'} \langle \xi, h_{d'} S^{-1}(h_d) \rangle \; \xi_d \; 
 \otimes \;\xi_{d'}.\]
 Here,  $\xi_d  \otimes \xi_{d'}$ is identified 
 with $(\xi_d \otimes 1_H) \otimes (1_A \otimes \xi_{d'})$ in $\Drin(H) \otimes R_H(A)$.
\end{example}

\section*{Acknowledgements} 
The authors would like to thank Gigel Militaru for his insights and for providing the references in  Remark~\ref{rem:GM}.
We are indebted to the anonymous referee for their careful work and many constructive suggestions, which helped us to improve the exposition of the paper.
R.~Laugwitz was supported by a Nottingham Research Fellowship. C.~Walton was partially supported by the Alexander Humboldt Foundation, and the US National Science Foundation grants \#DMS-2100756, 2348833. C. Walton was also hosted by the University of Hamburg during many phases of this project, and she would like to thank  her hosts for their great hospitality and for providing excellent working conditions.
M.~Yakimov was partially supported by the US National Science Foundation grant \#DMS-2200762.

%%%%%%%%%%%%%%%%%%%%%%%%%%%%%%%%%%%%%%%%%%%%%%%%%%%%%%%%%%%%%%%%%%%%%%%%%%%%

\bibliography{LWY-2022}
\bibliographystyle{amsrefs}%agsm or dcu

%%%%%%%%%%%%%%%%%%%%%%%%%%
%%%%%%%%%%%%%%%%%%%%%%%%%%
%%%%%%%%%%%%%%%%%%%%%%%%%%

\end{document}